\documentclass{article}

\usepackage[left=100pt,right=100pt,top=50pt,bottom=50pt]{geometry}

\usepackage{mathptmx}       
\usepackage{helvet}         
\usepackage{courier}        
\usepackage{type1cm}        
\usepackage{graphicx,color,framed}        
\usepackage[bottom]{footmisc}

\usepackage{amsmath,amsfonts,amsthm,enumerate}
\usepackage{authblk}

\newdimen\svparindent
\parindent\svparindent

\newenvironment{svgraybox}%
       {\fboxsep=12pt\relax
        \begin{shaded}%
        \list{}{\leftmargin=12pt\rightmargin=2\leftmargin\leftmargin=0pt\topsep=0pt\relax}%
        \expandafter\item\parindent=\svparindent
        \hskip-\listparindent}%
       {\endlist\end{shaded}}%

\definecolor{shadecolor}{gray}{.85}%
\definecolor{tintedcolor}{gray}{.80}%

\newtheorem{theorem}{Theorem} 
\newtheorem{lemma}[theorem]{Lemma} 
\newtheorem{proposition}[theorem]{Proposition} 
\newtheorem{definition}[theorem]{Definition} 
 
\newtheorem{remark}[theorem]{Remark} 

\numberwithin{equation}{section}
\numberwithin{theorem}{section}

\begin{document}

\title{Concentration waves of chemotactic bacteria: the discrete velocity case}
\author[1]{Vincent Calvez}
\author[2]{Laurent Gosse} 
\author[3]{Monika Twarogowska}
\date{}

\affil[1]{{\small CNRS \& Institut Camille Jordan, Universit\'e de Lyon, France, \texttt{vincent.calvez@math.cnrs.fr}}}
\affil[2]{{\small IAC, CNR, via dei Taurini, 19, 00185 Roma (Italia), 
\texttt{l.gosse@ba.iac.cnr.it}}}
\affil[3]{{\small Unit\'e de Math\'ematiques Pures et Appliqu\'ees, Ecole Normale Sup\'erieure de Lyon, and Inria, project-team NUMED, Lyon, France,  
\texttt{monika.twarogowska@ens-lyon.fr}}
}
%
%
\maketitle

\abstract{
The existence of travelling waves for a coupled system of hyperbolic/ parabolic equations is established in the case of a finite number of velocities in the kinetic equation. This finds application in collective motion of chemotactic bacteria. The analysis builds on the previous work by the first author (arXiv:1607.00429) in the case of a continuum of velocities. Here, the proof is specific to the discrete setting, based on the decomposition of the population density in special Case's modes. Some counter-intuitive results are discussed numerically, including the co-existence of several travelling waves for some sets of parameters, as well as the possible non-existence of travelling waves.}

\section{Introduction and framework}

This note is devoted to the analysis of kinetic models for travelling bands of chemotactic bacteria  {\em E. coli} in a microchannel. 
This builds on the series of papers \cite{saragosti_mathematical_2010,saragosti_directional_2011,calvez_confinement_2015,calvez_chemotactic_2016,companion}. This series of works is motivated by the following seminal experiment: 
A population of  bacteria {\em E. coli} is initially located on the left side of a microchannel after centrifugation (approximately  $5.10^5$ individuals). After short time, a significant fraction of the population moves towards the right side of the channel, at constant speed, within a constant profile \cite{saragosti_directional_2011}, see Figure \ref{fig:cartoon} for a schematic picture. 
We refer to \cite{adler_chemotaxis_1966} for the original experiment, and \cite{tindall_overview_2008} for a thorough review about the mathematical modelling of collective motion of bacteria in the light of this experiment, initiated by the celebrated work by Keller and Segel \cite{keller_traveling_1971}.

\begin{figure}[t]
\begin{center}
\includegraphics[width = .8\linewidth]{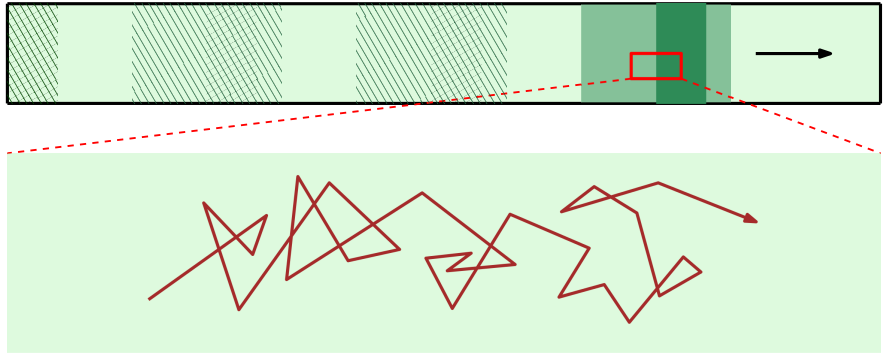}
\caption{Cartoon of concentration waves of bacteria as observed in experiments (see \cite{adler_chemotaxis_1966,saragosti_directional_2011}). The population of bacteria is initially located on the left hand side of the channel after centrifugation. Shortly, a large fraction of the population detaches and propagate to the right side at constant speed. Individual trajectories follows a run-and-tumble process in first approximation: cells alternate between straight runs and fast reorientation events (tumbles). The duration of run phases is modulated by sensing temporal variations of the chemical gradients in the environment. We refer to \cite{berg_e._2004} for biological aspects of motions of {\em E. coli}.}
\label{fig:cartoon}
\end{center}
\end{figure}


Kinetic models have proven to be well suited to study bacteria locomotion and chemotaxis, which navigate in a liquid medium according to a biased run-and-tumble process  \cite{berg_e._2004}. They were first introduced and investigated in  the 70's by Stroock \cite{stroock_stochastic_1974}, then in the early 80's by Alt \cite{alt_biased_1980}. We refer to \cite{othmer_models_1988,erban_signal_2005,dolak_kinetic_2005,chalub_model_2006,xue_macroscopic_2013,perthame_derivation_2015} for the description of the run-and-tumble model at multiple scales. In particular, \cite{xue_travelling_2010} and \cite{franz_travelling_2013} deals with the modelling of the same experiment, and \cite{almeida_existence_2014,emako_traveling_2016} is about the modelling of the interactions between two strains into the same wave of propagation. We also refer to the recent works   \cite{filbet_inverse_2013,rousset_simulating_2013,filbet_numerical_2014,yasuda_monte_2015,emako_well-balanced_2016} concerning numerical simulations of multiscale models of chemotactic bacteria.

Here, we investigate a basic kinetic model, coupled to reaction-diffusion equations for the dynamics of chemical species. 
The population of bacteria is described by its density $f(t,x,v)$ in the phase space position$\times$velocity. In addition, two chemical species are considered, according to the leading hypothesis in \cite{salman_solitary_2006,xue_travelling_2010,saragosti_mathematical_2010,saragosti_directional_2011}. We denote by $N(t,x)$ the concentration of some nutrient, which is distributed homogeneously in the domain at initial time. We also denote by $S(t,x)$ the concentration of some amino-acid which mediates cell-cell communication. The full model is written as follows,
\begin{equation}
\label{eq:meso model}
\left\{\begin{array}{l}
\displaystyle \partial_t f(t,x,v) + v\cdot \nabla_x f(t,x,v) =  \int_{v'\in V}  {\bf T} (t,x,v')  f(t,x,v')\, d\nu(v') -    {\bf T} (t,x,v) f(t,x,v)\medskip\\
\partial_t S(t,x)  =  D_S \Delta S(t,x) - \alpha S(t,x) + \beta \rho(t,x) \medskip\\
\partial_t N(t,x)  =  D_N \Delta N(t,x)  - \gamma \rho(t,x) N(t,x)\, ,
\end{array}\right.
\end{equation}
where $\rho$ denotes the spatial density: 
$\rho(t,x) = \int f(t,x,v)\, d\nu(v)\, . $
Here, $\alpha,\beta,\gamma,D_S,D_N$ are positive constants. The measure $\nu$ is a compactly supported probability measure on the velocity space. The tumbling (scattering) rate ${\bf T} (t,x,v)$ expresses temporal sensing of navigating bacteria:
\begin{multline*}
{\bf T}(t,x,v) = 1 - \chi_S \mathrm{sign}\left (\partial_t S(t,x)+ v\nabla_x S(t,x)\right ) -  \chi_N \mathrm{sign}\left (\partial_t N(t,x)+ v\nabla_x N(t,x)\right )\, , \\(\chi_S,\chi_N)\in (0,\frac12)\times(0,\frac12)\, .
\end{multline*} 
It is assumed that any single  bacteria is influenced by  temporal variations of both concentrations $S$ and $N$ along its trajectory with velocity $v$. Furthermore, it is able to distinguish perfectly between {\em favourable} directions (with positive variation) and {\em unfavourable} directions (with negative variation). It modulates the tumbling rate accordingly: runs are relatively longer if direction is favourable (because the tumbling rate is relatively smaller). Finally, we assume that both signal contributions are additive, with possibly two different values for the coefficients $\chi_S, \chi_N$. 

For the sake of simplicity, we assume that
\begin{equation}\label{eq:cond chi}
  \chi_N\leq \chi_S\, .
  \end{equation}  
This condition appears at some point during the analysis. We believe that our results hold true also in the opposite case $\chi_N>\chi_S$. However, this would require more complicated arguments that we postpone for future work.   

As the kinetic equation is conservative, we assume without loss of generality that 
\begin{equation}\label{eq:unit}
\iint f(t,x,v) \, d\nu(v)dx = \int \rho(t,x)\, dx = 1\, .
\end{equation}

We refer to \cite{saragosti_directional_2011,calvez_chemotactic_2016} for a detailed discussion about the relevance of this model. 

In this work, we examine the case of a finite number of velocities $N = 2K+1$, for some integer $K$. Let $(v_k)$ be the set of discrete velocities,  and $\boldsymbol\omega = (\omega_k)$ be the corresponding weights. 
We adopt the following notation: index $k$ ranges from $-K$ to $K$, with $v_0 = 0$. Let denote $\mathcal K = [-K,K]$ the set of indices. 
The measure  $\nu$ is defined as follows,
\begin{equation*}
\nu =  \sum_{k\in \mathcal K} \omega_k \delta_{v_k} \,,\quad \sum_{k\in \mathcal K} \omega_k = 1\,.
\end{equation*} 
We assume that nodes and weights are symmetric with respect to the origin: 
\begin{equation*}
(\forall k\in \mathcal K) \quad v_k = -v_{-k} \, , \quad \omega_k = \omega_{-k} \, . 
\end{equation*}


We seek one-dimensional  travelling wave solutions, that we write $f(x-ct,v)$, $S(x-ct), N(x-ct)$ with some slight abuse of notation. Thus, we are reduced to investigate the following problem: 
\begin{equation}\label{eq:TW}
\begin{cases}
  \displaystyle  (v_k-c)\partial_z f(z,v_k) = \sum_{k'\in \mathcal K} \omega_{k'} T(z,v_{k'}-c) f(z,v_{k'}) -   T(z,v_{k}-c) f(z,v_{k}) \medskip\\
 - c\partial_z S(z)  -  D_S \partial^2_{z} S(z)  + \alpha S(z) = \beta \rho(z) \medskip\\
- c\partial_z N(z)  - D_N \partial_{z}^2 N(z) = - \gamma \rho(z) N(z)\, , 
\end{cases}
\end{equation}
where the speed  $c$ is an unknown real number. As the problem is symmetric, we look for a positive value $c>0$ without loss of generality. The tumbling rate in the moving frame is
\begin{equation}\label{eq:T}
T(z,v-c) = 1 - \chi_S \mathrm{sign}\left ( (v-c)\partial_z S(z)\right ) -  \chi_N \mathrm{sign}\left ((v-c)\partial_z N(z)\right )\, .
\end{equation}

In \cite{calvez_chemotactic_2016}, the existence of travelling waves is established in the  case of a continuum of velocity. Namely, it is assumed that the measure $\nu$ is absolutely continuous with respect to the Lebesgue measure, and that the probability density function belongs to $L^p$ for some $p>1$. 
Here, we investigate this problem in the case of discrete velocities, which is not contained in  \cite{calvez_chemotactic_2016}. Rather than stating a global result, we present the framework for studying \eqref{eq:TW}. This enables to discuss numerically  the possible non existence of travelling waves. 

The following methodology is adopted in \cite{saragosti_mathematical_2010,calvez_chemotactic_2016}:
\begin{svgraybox}{\textbf{Framework: construction of travelling waves}}
\begin{enumerate}[(i)]
\item Assume  {\em a priori} that $N$ is increasing everywhere, and that $S$ is unimodal, with a single maximum located at $z = 0$, say. This enables to decouple the kinetic equation from the reaction-diffusion equations in \eqref{eq:TW}, since $T$ is then fully characterized by \eqref{eq:tumbling c}, see also Figure \ref{fig:tumbling}.
\item Prove that there exists a non trivial density $f$, which decays exponentially fast as $|z|\to +\infty$. This is an expression of the confinement effect due to the biased modulation of runs, see Section \ref{sec:confinement}.
\item  Prove that the spatial density $\rho$ is also unimodal, with a single maximum located at the transition point $z = 0$, as well. This is the hard task, because $f$ does not share this monotonicity property for all $v$, but $\rho$ does, as a consequence of compensations in averaging, see Section \ref{sec:mono}.
\item Check {\em a posteriori} that $N$ is increasing, and that $S$ is unimodal, with a single maximum reached at $z=0$. The former is unconditionally  true, provided that $c>0$. The latter condition is the crucial one which enables to prescribe the value of $c$.  
\end{enumerate}
\end{svgraybox}

\begin{figure}[t]
\begin{center}
\includegraphics[width = .8\linewidth]{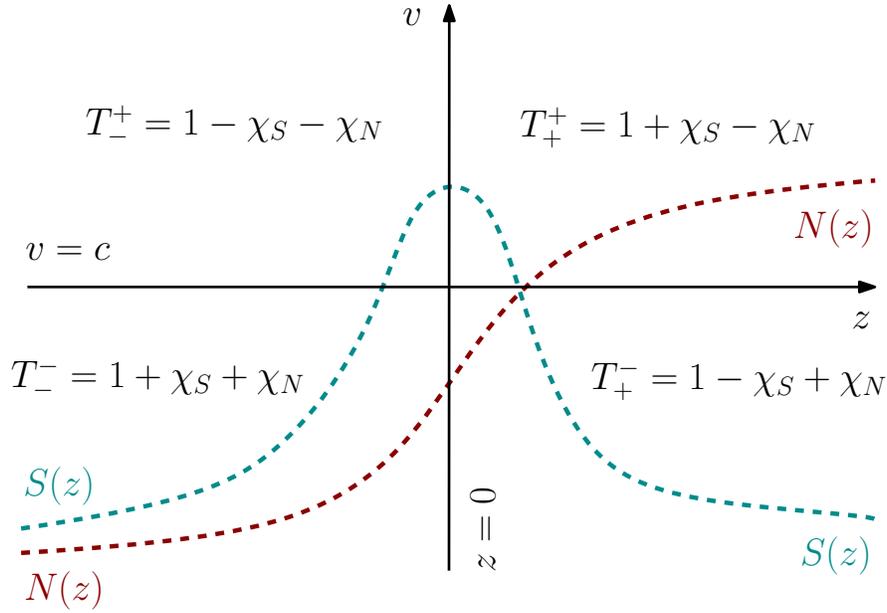}
\caption{The tumbling rate $T(z,v-c)$ at a glance, according to the rule of signes expressed in~\eqref{eq:sign T}.}
\label{fig:tumbling}
\end{center}
\end{figure}

Suppose that $N$ and $S$ share the appropriate monotonicity conditions, as in (i). Then, $T$ \eqref{eq:T} is prescribed as follows: 
\begin{equation}\label{eq:tumbling c}
T(z,v-c) = 1 + \chi_S \mathrm{sign}\left ( (v-c)z\right ) -  \chi_N \mathrm{sign}\left ((v-c)\right )\,.
\end{equation} 
We adopt the following short-cut notations (see also Figure \ref{fig:tumbling}): 
\begin{equation}\label{eq:sign T}
\begin{cases}
T_-^- = 1 + \chi_S + \chi_N\, , \quad \text{(direction is doubly unfavourable),} \medskip\\
T_-^+ = 1 - \chi_S - \chi_N\quad \text{(direction is doubly favourable),} \medskip\\
T_+^- = 1 - \chi_S + \chi_N \quad \text{(direction is favourable for $S$, but unfavourable for $N$),} \medskip\\
T_+^+ = 1 + \chi_S - \chi_N\quad \text{(direction is unfavourable for $S$, but favourable for $N$).}  
\end{cases}
\end{equation} 

Suppose that $c$ is given in a suitable interval for confinement purposes (see below). 
We can associate a probability density $f$, in a unique way, solution of the first equation in  \eqref{eq:TW}. This defines the concentration $S$, through the spatial density $\rho$, as the result of the elliptic equation in the second line of  \eqref{eq:TW}. Then, the matching condition in point (iv) can be formulated as $\partial_zS(0) = 0$. This motivates the following definition. 

\begin{definition} 
Let $c\mapsto \Upsilon(c)$ be the derivative of the concentration $S(z)$ at $z=0$: 
\begin{equation}\label{eq:upsilon}
\Upsilon(c) = \partial_z S(0)\, .
\end{equation}
\end{definition}

The ultimate goal of this paper is to prove that this definition makes perfect sense, and that solving the equation $\Upsilon = 0$ in the range of admissible wave speeds is equivalent to solving problem \eqref{eq:TW}. We also discuss some counter-intuitive examples for which the function $\Upsilon(c)$ has multiple roots, or have no admissible root. 

Section \ref{sec:confinement} is dedicated to the so-called confinement problem: being given appropriate monotonicity of $S$ and $N$, prove the existence of a unique normalized density function $f$ solution of the kinetic equation. Here, we follow a numerical analyst's viewpoint, by decomposing the solution into a finite sum of Case's special functions. Several properties of the solution are also established. Section \ref{sec:mono} justifies the framework presented above, as it is proven that the spatial density $\rho$ reaches a unique maximum, so does the chemical concentration $S$. A careful analysis of the shape of the velocity profiles is required there. Finally, three case studies are presented in Section \ref{sec:EX}, together with numerical simulations of the Cauchy problem.

\section{Confinement by biased velocity-jump processes}\label{sec:confinement}

We denote by $c_\star< c^\star$ resp. the infimum and the supremum of admissible velocities. Exact definitions are given below,  during the course of analysis, see \eqref{eq:def c+}-\eqref{eq:def c-}. Roughly speaking, if $c<c_\star$, the cell density is not confined on the right hand side. On the other hand, if $c>c^\star$, the cell density is not confined on the left-hand-side. 
Let $\mathcal C = (c_\star, c^\star)\setminus\{(v_k)_{k\in \mathcal K}\}$ be the set of admissible velocities. 

The  density of tumbling events per unit of time, which appears together with $\rho$ as a macroscopic quantity in \eqref{eq:TW}, will play a major role in the subsequent analysis: 
\begin{equation}\label{eq:def I}
I(z) = \sum_{k\in \mathcal K} \omega_{k} T(z,v_{k}-c) f(z,v_{k})\,,
\end{equation}

We have the following result, adapted from \cite{calvez_confinement_2015,calvez_chemotactic_2016}, but in a finite velocity setting.
\begin{theorem}\label{th:exist}
Let $c\in \mathcal C$. There exists a unique positive function $f$ with normalization \eqref{eq:unit}, such that for all $z\in \mathbb{R}$, and $k\in \mathcal{K}$,
\begin{equation}\label{eq:decoupled}
(v_k-c)\partial_z f(z,v_k) = \sum_{k'\in \mathcal K} \omega_{k'} T(z,v_{k'}-c) f(z,v_{k'}) -   T(z,v_{k}-c) f(z,v_{k})\, .
\end{equation} 
The functions $z\mapsto f(z,v_k)$ are exponentially decaying on both sides $z<0$ and $z>0$. In addition, we have the following asymptotic behaviour: there exist positive numbers $\lambda_+,\mu_+, \lambda_-, \mu_->0$, such that 
\begin{equation}\label{eq:as mon}
\begin{cases}
(\forall z>0)\quad \dfrac{dI}{dz}(z)= - \mu_+ \exp(-\lambda_+ z) \left (1 + o(z)\right ) \medskip\\
(\forall z<0)\quad \dfrac{dI}{dz}(z)= \mu_-  \exp(\lambda_- z) \left (1+ o(z)\right )
\end{cases}
\end{equation} 
\end{theorem} 

The dependency of the various constants with respect to the parameters is described in the proof. In particular, we keep track of the dependency with respect to the weights $\boldsymbol\omega$, as they will vary in Section \ref{sec:mono}. It is crucial to guarantee that the correction terms $o(z)$ in \eqref{eq:as mon} are uniformly small for large $|z|$. In particular, this requires that  the extremal weight $\omega_{K}$ is bounded below by some positive constant, see \eqref{eq:degeneracy} below. 

Note that a similar result in higher dimension was established recently in \cite{mischler_linear_2016}, with a different approach. In the latter work, the description of the stationary distribution $f$ is less explicit. 

\begin{proof}
The proof builds on the numerical analysis developed in \cite[Section 7]{calvez_chemotactic_2016}. It is very much inspired from the study of discrete Case's modes for linear kinetic transport equations, see \cite{gosse_computing_2013} and references therein.

We seek the solution as a combination of Case's modes on each side of the origin $z= 0$. For this purpose, we define the cutting index $J\geq 0$, such that
\begin{equation}\label{eq:cut}
v_{-K} < v_{-K+1} < (\dots) < v_{J} < c < v_{J+1} < (\dots) < v_K\, .
\end{equation}
We make the following ansatz:
\begin{equation}\label{eq:decomposition}
\begin{cases}
(\forall z<0)\,(\forall v_k)\quad&\displaystyle f(z,v_k) = \sum_{j = -K}^J a_j \dfrac{\exp(-\lambda_j(c)z)}{T_-(v_k - c) - \lambda_j(c) (v_k-c)}\medskip\\
(\forall z>0)\,(\forall v_k)\quad&\displaystyle f(z,v_k) = \sum_{j = J+1}^{K} b_j \dfrac{\exp(-\lambda_j(c)z)}{T_+(v_k - c) - \lambda_j(c) (v_k-c)}\, ,
\end{cases}
\end{equation}
where $(a_j,b_j)$ are unknown coefficients. 
It is immediate to check that each mode in \eqref{eq:decomposition} is indeed a special solution if, and only if, $\lambda_j(c)$ is a root of the following dispersion relation,
\begin{equation}\label{eq:dispersion}
\sum_{k\in\mathcal K} \omega_k \dfrac{T_\pm(v_k-c)}{T_\pm(v_k-c)-\lambda  (v_k - c) } = 1\,.
\end{equation}
This equation possesses the trivial solution $\lambda = 0$. However, it is excluded since we seek solutions which are integrable over the whole line. Otherwise, \eqref{eq:dispersion} is equivalent to
\begin{equation}\label{eq:dispersion bis}
\sum_{k\in\mathcal K} \omega_k\left( \dfrac{T_\pm(v_k-c)}{v_k-c}-\lambda\right  )^{-1}= 0\,.
\end{equation}
The latter equation has exactly $K+J+1$ negative solutions, associated with $T_-$, which are interlaced as follows, 
\begin{multline*}
-\infty < \dfrac{T_-^-}{v_{J}-c} < \lambda_{J}(c) < \dfrac{T_-^-}{v_{J-1}-c} < \lambda_{J-1}(c) < \dfrac{T_-^-}{v_{J-2}-c} \\ < (\dots) < \lambda_{-K-1}(c) < \dfrac{T_-^-}{v_{-K}-c} < \lambda_{-K}(c) < 0\, .
\end{multline*}
This definition of the exponents $(\lambda_j(c))$ deserves some careful explanation. The fact that there exists a negative root between  the last singular value $\frac{T_-^-}{v_{-K}-c}$ and value 0 is a consequence of the increasing monotonicity of \eqref{eq:dispersion bis} with respect to $\lambda$, and of the evaluation at $\lambda = 0$, namely
\begin{equation*}
\sum_{k\in\mathcal K} \omega_k \dfrac{v_k-c}{T_-(v_k-c)}>0\,.
\end{equation*}
The latter expresses the fact that the mean algebraic run length is positive, which is obviously required for the confinement phenomenon. This defines an upper value for $c$: $c<c^\star$, where 
\begin{equation}\label{eq:def c+}
\sum_{k\in\mathcal K} \omega_k \dfrac{v_k-c^\star}{T_-(v_k-c^\star)}=0\, .
\end{equation}
Intuitively, the speed $c$ cannot be too large, in order to ensure confinement on the left side. Obviously, in the extreme case where $c> v_K$, relative speeds $v_j-c$ are all negative, so there is no possible confinement!

Similar conclusions can be drawn on the positive side $z>0$: \eqref{eq:dispersion bis} has $K-J$ positive solutions, associated with $T_+$, which are interlaced as follows,
\begin{multline}
\label{eq:entrelacing+}
0 < \lambda_{K}(c) < \dfrac{T_+^+}{v_{K}-c} < \lambda_{K-1}(c) < \dfrac{T_+^+}{v_{K-1}-c}  \\ < (\dots) < \lambda_{J+2}(c) < \dfrac{T_+^+}{v_{J+2}-c} <  \lambda_{J+1}(c) < \dfrac{T_+^+}{v_{J+1}-c}  < + \infty\, .
\end{multline}
Again, the existence of a positive root below the singular value $\frac{T_+^+}{v_{J+1}-c}$ is guaranteed if, and only if, the mean algebraic run length is negative, namely
\begin{equation}\label{eq:confinement +}
\sum_{k\in\mathcal K} \omega_k \dfrac{v_k-c}{T_+(v_k-c)}<0\,.
\end{equation}
The latter prescribes a lower value for $c$: $c>c_\star$, where 
\begin{equation}\label{eq:def c-}
\sum_{k\in\mathcal K} \omega_k \dfrac{v_k-c_\star}{T_+(v_k-c_\star)}=0\, .
\end{equation}

\begin{remark}
It is a consequence of condition \eqref{eq:cond chi} that $c_\star \leq 0$. Indeed, the mean algebraic run length is non positive when $c = 0$. As we seek wave solutions travelling to the right side, it is legitimate to restrict to $c\in(0,c^\star)$ in the present work. 
\end{remark}

In a second step, we associate the number of degrees of freedom in \eqref{eq:decomposition} with the incoming data on each side of the origin. From \eqref{eq:cut}, we deduce that there are $K+J+1$ negative relative velocities $(v_j-c)_{j\in[-K,J]}$, and $K-J$ positive relative velocities $(v_j-c)_{j\in[J+1,K]}$. As a consequence, the solution on the left side $f(z,v)|_{z<0}$ is characterized by the incoming data at the origin\footnote{We refer to \cite{calvez_confinement_2015,calvez_chemotactic_2016} for a discussion between this characterization and the Milne problem in radiative transfer theory \cite{bardos_diffusion_1984}.}, {\em i.e.} the vector 
\begin{equation*}
G_- = \begin{pmatrix}
f(0,v_{-K})\\ \vdots \\ f(0,v_{J}) 
\end{pmatrix}
\end{equation*}
On the other hand, the  solution on the right side $f(z,v)|_{z>0}$ is characterized by the incoming data at the origin, {\em i.e.} the vector 
\begin{equation*}
G_+ = \begin{pmatrix}
f(0,v_{J+1})\\ \vdots \\ f(0,v_{K}) 
\end{pmatrix}
\end{equation*}
We may supposedly relate the incoming data, and the degrees of freedom in the decompositions \eqref{eq:decomposition}, by  square matrices, as they are in the same number. Alternatively, we can express compatibility conditions in \eqref{eq:decomposition} at $z=0$ in the form of a transfer operator. For this we must identify both decompositions at $z=0$:
\begin{equation}\label{eq:transfer}
(\forall v_k)\quad\displaystyle  \sum_{j = -K}^J \dfrac{a_j }{T_-(v_k - c) - \lambda_j(c) (v_k-c)}= \sum_{j = J+1}^{K}  \dfrac{b_j}{T_+(v_k - c) - \lambda_j(c) (v_k-c)}
\end{equation}
This is exactly equivalent to developing a fixed point argument on the right inflow data $G_+$ as in \cite{calvez_confinement_2015,calvez_chemotactic_2016}: suppose we are given $G_+$, we can invert the second part of \eqref{eq:decomposition} to find $(b_j)_{j\in[J+1,K]}$ by solving a $(K-J)\times (K-J)$ linear system. This prescribes a left inflow data $G_-$. Again, we can invert the first  part of \eqref{eq:decomposition} to find $(a_j)_{j\in[-K,J]}$ by solving a $(K+J+1)\times (K+J+1)$ linear system. This yields in turn a right inflow data $\widetilde{G}_+$ that should coincide with $G_+$. This fixed point procedure is all contained in \eqref{eq:transfer}. The existence of a non zero vector such that  \begin{equation}\label{eq:transfer matrix}
\begin{pmatrix}
\left ( T_-(v_k - c) - \lambda_j(c) (v_k-c) \right )^{-1} \\
- \left ( T_+(v_k - c) - \lambda_j(c) (v_k-c) \right )^{-1} 
\end{pmatrix}_{(k,j)}
\begin{pmatrix}
a_j\\ b_j 
\end{pmatrix}_{(j)} = 0 \, , 
\end{equation}
is a direct consequence of mass conservation. Indeed, formulation \eqref{eq:dispersion bis} can be rewritten as 
\begin{equation}\label{eq:dispersion ter}
(\forall j)\quad \sum_{k\in\mathcal K} \omega_k\dfrac{v_k-c}{ T_\pm(v_k - c) - \lambda_j(c) (v_k-c)} = 0\,.
\end{equation}
This means exactly that the row vector $(\omega_k(v_k-c))_{(k)}$ belongs to the kernel of the adjoint problem of \eqref{eq:transfer matrix}. We deduce the existence of non trivial coefficients $(a_j,b_j)$. This yields a solution of the stationary problem \eqref{eq:decoupled} over the whole line. 

Positivity and uniqueness of $f$ are both consequences of the ergodicity underlying the linear (decoupled) kinetic equation, see \cite{calvez_confinement_2015,calvez_chemotactic_2016}. 

The rest of the proof of Theorem \ref{th:exist} consists in a series of Lemma that establish appropriate bounds for the solution. This leads ultimately to the quantitative estimate \eqref{eq:as mon}. 

Contrary to the decomposition of the solution in Case's modes \eqref{eq:decomposition}, we shall now use the Duhamel formulation along characteristic lines:
\begin{align}
& (\forall z>0)(\forall v_k<c) \quad f(z,v_k) = \int_0^{+\infty} I(z-s(v_k-c)) \exp\left ( -s T_+^- \right )\, ds \label{eq:duhamel1} \\
& (\forall z>0)(\forall v_k>c) \quad f(z,v_k) = f(0,v_k)\exp\left (  -\frac{ z T_+^+}{v_k-c} \right )\nonumber\\
& \qquad\qquad\qquad  \qquad\qquad\qquad +  \int_0^{\frac{z}{v_k-c}} I(z-s(v_k-c)) \exp\left ( -s T_+^+ \right )\, ds\, . \label{eq:duhamel2}
\end{align}
Similar formulas hold on the other side $(z<0)$.

\begin{lemma}[$L^\infty$ bound]
The function $f$ is uniformly bounded,  independently of the weigth $\boldsymbol\omega$. 
\end{lemma}

\begin{proof}
Firstly, 
let recall that the solution $f$ is normalized to have unit mass \eqref{eq:unit}:
\begin{equation*}
\int_{\mathbb{R}}  \rho(z) \, dz = \int_{\mathbb{R}} \sum_{k\in \mathcal K} \omega_k f(z,v_k)\, dz = 1\, .
\end{equation*}
As a by-product, the macroscopic quantity $I(z)$ \eqref{eq:def I}, which is easily comparable with $\rho(z)$, has a uniformly bounded integral over $\mathbb{R}$.
We immediately deduce from the Duhamel formula \eqref{eq:duhamel1}, that $f(z,v_k)$ is uniformly bounded for $z>0$ and $v_k<c$:
\begin{align}
f(z,v_k) & = \int_0^{+\infty} I(z+y) \exp\left (  -\frac{ y T_+^+}{c-v_k} \right )\dfrac{1}{c-v_k} \, dy\nonumber\\
& \leq \dfrac{1}{c-v_k} \int_{\mathbb{R}} I(z)\, dz\nonumber \\
& \leq \dfrac{\max T}{c-v_k} \, .\label{eq:bound fz} 
\end{align}
Similar estimate holds true for $z<0$ and $v_k>c$. Therefore, $f(0,v_k)$ is uniformly bounded for all $k$, independently of the weigth  $\boldsymbol\omega$. 

The same bounds can be propagated to any $z>0$ (resp. $z<0$) using \eqref{eq:duhamel2}. 
%
\qed\end{proof}

The decomposition \eqref{eq:decomposition} is a nice characterization of the density over $\{z<0\}$ and $\{z>0\}$, respectively. The coefficients $(a_j,b_j)$ were not specified in \eqref{eq:decomposition}. However, they are in relation with the profile at $z = 0$ as in the  following lemma.

\begin{lemma}[Expression of $b_i$'s]\label{lem:ortho}
Each coefficient $b_i$ in \eqref{eq:decomposition} is given by the following orthogonality formula: 
\begin{equation}\label{eq:ortho}
b_i = \dfrac{\displaystyle\sum_{k\in \mathcal K} \omega_k (v_k-c) f(0,v_k) \dfrac{T_+(v_k-c)}{T_+(v_k-c) - \lambda_i(c) (v_k-c)}}{\displaystyle\sum_{k\in \mathcal K}  \omega_k(v_k-c) \dfrac{T_+(v_k-c)}{\left (T_+(v_k-c) - \lambda_i(c) (v_k-c)\right )^2}}\, .
\end{equation}
\end{lemma}

\begin{proof}
We take the scalar product of the second line in \eqref{eq:decomposition} and the dual eigenvector  indexed by $i$, $\frac{T_+(v_k-c)}{T_+(v_k-c) - \lambda_i(c) (v_k-c)}$, with the weights $ (\omega_k(v_k-c))$\footnote{Notice that these weights are  signed.}. We realize that crossed terms cancel for $j\neq i$:
\begin{align*}
& \sum_{k\in \mathcal K} \omega_k (v_k-c) \dfrac{T_+(v_k-c)}{\left (T_+(v_k-c) - \lambda_j (v_k-c)\right )\left (T_+(v_k-c) - \lambda_i  (v_k-c)\right )} \\
& = \dfrac{1}{\lambda_i - \lambda_j } \sum_{k\in \mathcal K} \omega_k  \dfrac{\left (\lambda_i - \lambda_j \right )(v_k-c) + T_+(v_k-c) - T_+(v_k-c)}{\left (T_+(v_k-c) - \lambda_j  (v_k-c)\right )\left (T_+(v_k-c) - \lambda_i  (v_k-c)\right )}T_+(v_k-c)\\
& = \dfrac{1}{\lambda_i  - \lambda_j } \left( - \sum_{k\in \mathcal K} \omega_k \dfrac{T_+(v_k-c)}{T_+(v_k-c) - \lambda_j  (v_k-c)} + \sum_{k\in \mathcal K}  \omega_k\dfrac{T_+(v_k-c)}{T_+(v_k-c) - \lambda_i  (v_k-c)}  \right )\\
& = 0\, . 
\end{align*}
The cancellation holds true due to the dispersion relation \eqref{eq:dispersion} which is common to $\lambda_i$ and $\lambda_j$.  Hence, only the contribution indexed by $i$  remains after multiplication. This yields \eqref{eq:ortho}. 
\qed\end{proof}

Next, we can rewrite the identity \eqref{eq:ortho} in a better way, with positive weights, instead of signed weights, by using the zero flux condition $\sum (v_k-c) f(0,v_k) = 0$, and \eqref{eq:dispersion ter}:
\begin{align}
b_i & = \dfrac{\displaystyle\sum_{k\in \mathcal K} \omega_k (v_k-c) f(0,v_k) \dfrac{\lambda_i(c)(v_k-c) }{T_+(v_k-c) - \lambda_i(c) (v_k-c)}}{\displaystyle\sum_{k\in \mathcal K} \omega_k (v_k-c) \dfrac{\lambda_i(c)(v_k-c)}{\left (T_+(v_k-c) - \lambda_i(c) (v_k-c)\right )^2}} \nonumber \\
 &  = \dfrac{\displaystyle\sum_{k\in \mathcal K}  \omega_k f(0,v_k) \dfrac{ (v_k-c)^2 }{T_+(v_k-c) - \lambda_i(c) (v_k-c)}}{\displaystyle\sum_{k\in \mathcal K} \omega_k \dfrac{ (v_k-c)^2}{\left (T_+(v_k-c) - \lambda_i(c) (v_k-c)\right )^2}} \, .
 \label{eq:ortho 2}
\end{align}
%
Interestingly, this reformulation suggests to use another scalar product, in order to derive appropriate bounds for the coefficients $(b_i)$. This is the purpose of the next lemma.  

\begin{lemma}[Bound of $b_i$'s]
Each coefficient $b_i$ in \eqref{eq:decomposition} is bounded in the following way:
\begin{equation}\label{eq:bound omega b}
(\forall i \in[J+1,K])\; (\forall k\in \mathcal K)\quad  \dfrac{|b_i||v_k-c|}{\left |T_+(v_k-c) - \lambda_i(c) (v_k-c)\right |} \leq \dfrac{\max T}{\sqrt{\omega_k} } \,.
\end{equation}  
\end{lemma}

\begin{proof}
For any $z>0$, let compute the weighted $\ell^2$ norm of $f(z,v_k)$ as follows,
\begin{align*}
& \|f(z,v_k)\|^2_{\ell^2(\omega_k(v_k-c)^2)}  =\sum_{k\in \mathcal K} \omega_k (v_k-c)^2 f(z,v_k)^2\\
& = \sum_{i,j\in[J+1,K]^2} b_i b_j \exp(-(\lambda_i(c) + \lambda_j(c))z )   \\
& \qquad\quad \times\sum_{k\in \mathcal K} \omega_k  (v_k-c) \dfrac{(v_k-c)}{\left (T_+(v_k-c) - \lambda_j(c) (v_k-c)\right )\left (T_+(v_k-c) - \lambda_i(c) (v_k-c)\right )} \\
& = \sum_{i\in[J+1,K]} b_i^2 \exp(-2\lambda_i(c)z ) \sum_{k\in \mathcal K} \omega_k \dfrac{(v_k-c)^2}{\left (T_+(v_k-c) - \lambda_i(c) (v_k-c)\right )^2}\, , 
\end{align*}
because the cross terms vanish, exactly as in the proof of Lemma \ref{lem:ortho}.

We deduce from this estimate at $z=0$, that for all $i \in[J+1,K]$, and all $k\in \mathcal K$, we have
\begin{equation*}
\sqrt{\omega_k} \dfrac{|b_i||v_k-c|}{\left |T_+(v_k-c) - \lambda_i(c) (v_k-c)\right |} \leq\max_{k'\in \mathcal K}\left (  f(0,v_{k'})|v_{k'}-c|\right )\, . 
\end{equation*}  
This yields \eqref{eq:bound omega b} by \eqref{eq:bound fz}. 
\qed\end{proof}

A fruitful consequence of the latter estimate concerns the asymptotic behaviour of $f$, as $z\to +\infty$. 
%
%
%
Characterization \eqref{eq:decomposition} is now of  great interest. Indeed, the large space behaviour is characterized by the slowest mode \eqref{eq:entrelacing+}. Loosely speaking, we have:
\begin{equation}\label{eq:asympt 1}
(\forall v_k)\quad f(z,v_k) \underset{z\to +\infty}{\sim}  \dfrac{b_K\exp(-\lambda_K(c) z )}{T_+(v_k - c) - \lambda_K(c) (v_k-c)}\,. 
\end{equation}
Furthermore, the asymptotic monotonicity is also intuitively clear from this decomposition: 
\begin{equation}\label{eq:asympt 2}
(\forall v_k)\quad \partial_z f(z,v_k) \underset{z\to +\infty}{\sim}  \dfrac{- \lambda_K(c) b_K\exp(-\lambda_K(c) z )}{T_+(v_k - c) - \lambda_K(c) (v_k-c)}\,. 
\end{equation}

For upcoming purposes, it is necessary to gain some quantitative control about 
the prefactors in the r.h.s. of \eqref{eq:asympt 1} and \eqref{eq:asympt 2}. This is the aim of the next lemma.

%



%
%
%

\begin{lemma}[Bound of $b_K$ from below]\label{lem:bK}
There exists a constant $C$, depending on $\omega$, such that $b_K\geq 1/C$. 
\end{lemma}

\begin{proof}
We extend easily the orthogonality formula \eqref{eq:ortho 2} to any $z>0$:
\begin{equation*}
b_i \exp(-\lambda_i(c)z) = \sum_{k\in \mathcal K}\left [ f(z,v_k)\left ( T_+(v_k-c) - \lambda_i(c) (v_k-c)\right ) \right ]\varphi_{k+}(\lambda_i(c))\, , 
\end{equation*}
where the weights $\varphi_{k+}(\lambda_i(c))>0$ are such that $\sum \varphi_{k+} = 1$: 
\begin{multline*}
\varphi_{k+} = \dfrac1{Z_+}  \dfrac{ (v_k-c)^2 }{\left (T_+(v_k-c) - \lambda_i(c) (v_k-c)\right )^2} \, , \\ 
Z_+ = \displaystyle\sum_{k'\in \mathcal K} \omega_{k'} \dfrac{ (v_{k'}-c)^2}{\left (T_+(v_{k'}-c) - \lambda_i(c) (v_{k'}-c)\right )^2}\, .
\end{multline*}
For $i=K$, we deduce, after integration over $\mathbb{R}_+$, that 
\begin{equation}\label{eq:bound bK below}
\dfrac{b_K}{\lambda_K(c)} \geq \left ( \min_{k\in \mathcal K} \dfrac{T_+(v_k-c) - \lambda_K(c) (v_k-c)}{\omega_k} \right ) \int_{\mathbb{R}_+} \rho(z)\, dz\, . 
\end{equation}
Here, we have used the peculiar property of index $K$, which is such  that $\lambda_K(c) (v_k-c) < T_+(v_k-c)$ for all $k\in \mathcal K$ \eqref{eq:entrelacing+}. The minimum value is uniformly bounded below by \eqref{eq:dispersion}. We could conclude from the unit mass normalization, provided that the integral would be taken over $\mathbb{R}$ in   \eqref{eq:bound bK below}. To overcome this small issue, it is necessary to connect both sides $\{z>0\}$ and $\{z<0\}$ at some point. This is a consequence of the uniform comparison between the two values $b_K$ and $a_{-K}$. To this end, we notice that the identity \eqref{eq:ortho 2} can be recast, for $i=K$, as
\begin{equation}\label{eq:l1norm}
b_K = \left \| f(0,v_k) \left ( T_+(v_k-c) - \lambda_K(c) (v_k-c)\right )   \right \|_{\ell^1(\varphi_+)} \, .
\end{equation}
On the other side, we have similarly
\begin{equation}\label{eq:l1norm-2}
a_{-K} = \left \| f(0,v_k) \left ( T_-(v_k-c) - \lambda_{-K}(c) (v_k-c)\right )   \right \|_{\ell^1(\varphi_-)} \, ,
\end{equation}
with the appropriate choice   for the probability weights $\varphi_-$:
\begin{multline*}
\varphi_{k-} = \dfrac{1}{Z_-} \dfrac{\omega_k  (v_k-c)^2}{\left (T_-(v_k-c) - \lambda_{-K}(c) (v_k-c)\right )^2}\, , \\  Z_- = \sum_{k\in \mathcal K} \dfrac{\omega_k  (v_k-c)^2}{\left (T_-(v_k-c) - \lambda_{-K}(c) (v_k-c)\right )^2}\, .
\end{multline*}
To compare the two $\ell^1$ norms \eqref{eq:l1norm} and \eqref{eq:l1norm-2}, we shall establish the following bounds: there exists a constant $C>0$, depending on $\boldsymbol\omega$ such that 
\begin{equation*}
(\forall k)\quad \dfrac1C \leq \dfrac{\left (T_+(v_k-c) - \lambda_K(c) (v_k-c)\right )\varphi_{k+}}{\left (T_-(v_k-c) - \lambda_{-K}(c) (v_k-c)\right )\varphi_{k-}} \leq C\, . 
\end{equation*}
On the one hand, it is easy to bound the quantities $T_\pm(v_k-c) - \lambda_{\pm K}(c) (v_k-c)$ from above and from below, uniformly with respect to $k$: 
In fact, it is sufficient to control uniformly the following piece of estimate: 
\begin{equation}\label{eq:hlyupim}
  \lambda_{K}(c) < \dfrac{T_+^+}{v_{K}-c}\, .
\end{equation} 
But, we deduce from \eqref{eq:dispersion} that 
\begin{equation}\label{eq:degeneracy}
 \omega_K \dfrac{T_+^+}{T_+^+ -  \lambda_{K}(c) (v_{K}-c)} \leq 1\quad \Rightarrow\quad T_+^+ -  \lambda_{K}(c) (v_{K}-c) \geq \omega_K T_+^+\, .
 \end{equation} 
This gives the required estimate. 

\begin{remark}
Importantly, the bound from below in \eqref{eq:degeneracy} degenerates only if $\omega_K$ vanishes. The fact that it degenerates as $\omega_K$ vanishes is quite obvious, as the velocities $v_{\pm K}$ effectively disappear from the problem. The limit system keeps at most $2K-1$ velocities. The crucial point is that it does not depend upon the other weights. 
\end{remark} 

On the other hand, we have
\begin{equation*}
 \dfrac{\varphi_{k+}}{\varphi_{k-}} = \dfrac{Z_-}{Z_+} \dfrac{\left (T_+(v_k-c) - \lambda_K(c) (v_k-c)\right )^2}{\left (T_-(v_k-c) - \lambda_{-K}(c) (v_k-c)\right )^2}\, .
 \end{equation*} 
It is a consequence of \eqref{eq:degeneracy} that all terms in this fraction are bounded from above and below. 

Finally, combining both estimates \eqref{eq:bound bK below} and the similar estimate for $\{z<0\}$:
\begin{equation*}
\dfrac{a_{-K}}{\lambda_{-K}(c)} \geq \left ( \min_{k\in \mathcal K} \dfrac{T_-(v_k-c) - \lambda_{-K}(c) (v_k-c)}{\omega_k} \right ) \int_{\mathbb{R}_-} \rho(z)\, dz\, , 
\end{equation*}
together with the normalization $\int_{\mathbb{R}_-} \rho(z)\, dz + \int_{\mathbb{R}_+} \rho(z)\, dz = 1$, we get that there exists a constant $C$, depending on $\omega_{K}$, such that 
\begin{equation*}
\max\left (b_K,a_{-K}\right ) \geq \dfrac1C\, .
\end{equation*}
Since $b_K$, and $a_{-K}$ are uniformly comparable by previous estimates, we conclude that they are both bounded from below  by some positive constant.  
\qed\end{proof}

Finally, we can combine all these estimates to determine the asymptotic behaviour of the derivative of $I$. 

\begin{lemma}[Asymptotic monotonicity of $I$]\label{lem:asympt mon}
The derivative of the macroscopic quantity $I$ satisfies the following quantitative estimate:
\begin{equation}
(\forall z>0)\quad \dfrac{dI}{dz}(z)= -  \lambda_K(c)  b_K  \exp(-\lambda_+(c) z) \left (1 + o(z)\right )\, .
\label{eq:quantitative}
\end{equation}
\end{lemma}
\begin{proof}
For $z>0$, we have,
\begin{align*}
& \dfrac{dI}{dz}(z) =   \sum_{k\in \mathcal K} \omega_{k} T_+(v_{k}-c)\partial_z f(z,v_{k}) \\
& = - \sum_{k\in \mathcal K} \omega_{k} T_+(v_{k}-c) \sum_{j = J+1}^{K} b_j \lambda_j(c) \dfrac{\exp(-\lambda_j(c) z)}{T_+(v_k - c) - \lambda_j(c) (v_k-c)} \\
& = - \lambda_K(c)  b_K \exp(-\lambda_K(c) z) \sum_{k\in \mathcal K }  \omega_{k} \dfrac{ T_+(v_{k}-c)}{T_+(v_k - c) - \lambda_K(c) (v_k-c)} \\
& \quad + K(K-J) \exp(-\lambda_{K-1}(c) z) \max_{(k,j)\in\mathcal K\times [J+1,K-1]} \left | \dfrac{\omega_k b_j \lambda_j(c) T_+(v_k-c)}{T_+(v_k - c) - \lambda_j(c) (v_k-c)} \right |\, .
\end{align*}
We deduce from the identity satisfied by $\lambda_K(c)$ \eqref{eq:dispersion} (for the first contribution), and from the bound \eqref{eq:bound omega b} together with the trivial inequality $\omega_k\leq \sqrt{\omega_k}$  (for the second contribution), that
\begin{equation*}
\dfrac{dI}{dz}(z)= - \lambda_K(c)  b_K \exp(-\lambda_K(c) z) + \mathcal O\left ( \exp(-\lambda_{K-1}(c) z) \right )\, .
 \end{equation*} 
The estimate \eqref{eq:quantitative} follows since $\lambda_K(c)$ and $b_K$ are positive, and also $\lambda_{K-1}(c)>\lambda_{K}(c)$ \eqref{eq:entrelacing+}.
\qed\end{proof}
This lemma concludes the proof of Theorem \ref{th:exist}. Before we move to the next Section, let us make the following important observations: 
\begin{enumerate}
\item The bound $b_K>0$ was carefully analysed in Lemma \ref{lem:bK}. It is uniform provided that $\omega_K$ is bounded below by some positive constant. The same observation holds true for the separation estimate $\lambda_{K-1}(c)>\lambda_{K}(c)$, by \eqref{eq:entrelacing+}-\eqref{eq:hlyupim}-\eqref{eq:degeneracy}
\item The bound $\lambda_K(c)>0$ will be analysed carefully in Lemma \ref{lem:confinement} below. 
\end{enumerate}
\qed\end{proof}

\section{Monotonicity of the spatial density}\label{sec:mono}

An important issue related to the existence of travelling waves for the coupled problem \eqref{eq:TW} is the monotonicity of the spatial density $\rho$. In this section, we establish the following result, without any restriction on the parameters, except \eqref{eq:cond chi}. 

\begin{theorem}\label{th:mono}
Let $c\in \mathcal C$. Let $f$ be the function defined in Theorem \ref{th:exist}. The spatial density $\rho(z) = \sum_{k\in \mathcal K} \omega_{k}   f(z,v_{k})$ changes monotonicity at $z=0$: it is increasing for $z<0$, and decreasing for $z>0$.  
\end{theorem}

Before presenting the proof of this statement, let us comment on the strategy. We present here the adaptation of the argument proposed in \cite{calvez_chemotactic_2016}, but in the case of finite velocities. The version developed here is much simpler because many regularity issues can be overcome. However, we insist on the fact that \cite{calvez_chemotactic_2016} does not readily contain the finite velocity case. 

\begin{figure}[t]
\begin{center}
\includegraphics[width = 0.8\linewidth]{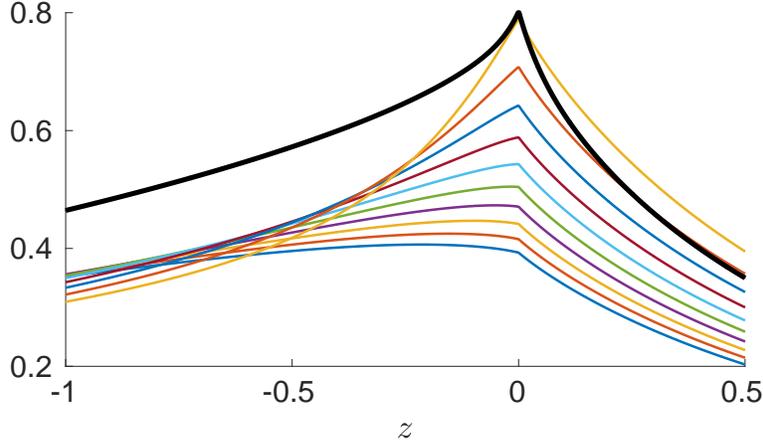}
\caption{\textbf{(Overshoot phenomena).} The spatial density is plotted in bold (black) line. It changes monotonicity at $z=0$, as claimed in Theorem \ref{th:mono}. The various functions $z\mapsto f(z,v_k)$ are plotted in coloured  lines for each $v_k<0$. For the most negative velocities, the maximum of the function is reached at some negative $z$. Here, $\chi_N= 0.2, \chi_S = 0.48$, and $c = 0.25$.}
\label{fig:overshoot}
\end{center}
\end{figure}

\begin{itemize}
\item
The first observation is that the monotonicity claimed in Theorem \ref{th:mono} is unlikely to be a straightforward consequence of the eigenmode decompositions \eqref{eq:decomposition}, as all terms are changing signs\footnote{Moreover, we have no argument so far to determine the signs of the coefficients $(a_j,b_j)$, if they have any.}. More strikingly, the functions $z\mapsto f(z,v_k)$ does not have the required monotonicity for fixed $v_k$, see Figure \ref{fig:overshoot}. Only the velocity average $\rho(z) = \sum \omega_k f(z,v_k)$ possesses the appropriate monotonicity.  
\item
As an alternative, we establish that the spatial density $\rho$ cannot change monotonicity as the weights $(\omega_k)$ vary continuously. This procedure is initialized with the case of two velocities only (all weights $\omega_k$ are set to zero except the extremal ones $\omega_K$). Monotonicity is obvious in the latter case since the solution is explicit as the concatenation of decaying exponential functions on each side \eqref{eq:decomposition}.  
\item
The key argument, that we called enhancement of monotonicity, resembles a maximum principle. Essentially, we prove that, if monotonicity of $\rho$ is appropriate, but in the large  (non decreasing for $z<0$, and non increasing for $z>0$), then monotonicity is necessarily strict, as stated in Theorem \ref{th:mono}. The proof goes through a refined description of the velocity distribution for any $z$. This central argument (Lemma \ref{lem:mon2}) is complemented with various compactness estimates, and some connectedness argument.  
\end{itemize}

\begin{proof}
Let denote $\mathcal K^+ = \{ k: v_k>c \}$ the set of positive relative velocities, and $\mathcal K^- = \{ k: v_k<c \}$ the set of negative relative velocities. We introduce 
\begin{equation*}
\begin{cases}
\rho^+(z) = \sum_{k\in \mathcal K^+} \omega_k f(z,v_k) \,,\medskip\\
\rho^-(z) = \sum_{k\in \mathcal K^-} \omega_k f(z,v_k) \,,\end{cases}
\end{equation*} 
as being  the contributions of positive and negative relative velocities to the spatial density, respectively. Finally, we recall the definition of the other macroscopic quantity of interest, which appears in the kinetic equation \eqref{eq:TW}: 
\begin{equation}\label{eq:def I 2}
I(z) = \sum_{k\in \mathcal K} \omega_{k} T(z,v_{k}-c) f(z,v_{k})\,.
\end{equation}
We make two key observations: 
\begin{itemize}
\item
Firstly, we have the following elementary reconstruction
\begin{equation*}
I(z) = \begin{cases}
T_-^- \rho^-(z) + T_-^+ \rho^+(z)  & \quad \text{for $z<0$}\, ,\medskip\\
T_+^- \rho^-(z) + T_+^+ \rho^+(z)  & \quad \text{for $z>0$}\, .
\end{cases}
\end{equation*}
Hence, identical monotonicity of both $\rho^+$ and $\rho^-$ implies the same monotonicity for $I$. 
\item
Secondly, the Duhamel formulation along characteristic lines \eqref{eq:duhamel1}-\eqref{eq:duhamel2} enables to reconstruct the kinetic density $f$ from the spatial density $I$.
We deduce the following important information from \eqref{eq:duhamel1} for negative relative velocities: \textbf{If $I$ is non increasing for $z>0$, and not constant\footnote{This is clearly the case due to integrability over $\mathbb{R}$.}, then $z\mapsto f(z,v_k)$ is decreasing for $z>0$ and $v_k<c$.}\\ 
As a by-product, we also deduce that 
$f(z,v_k)$ is increasing with respect to velocity on $\mathcal K^-$:
\begin{equation}\label{eq:cas -}
(\forall (v_i,v_j)\in \mathcal K^-\times \mathcal K^-) \quad (v_i < v_j) \Rightarrow f(z,v_i) < f(z,v_j) < \dfrac{I(z)}{T_+^-} \, . 
\end{equation}
\end{itemize}
Similar result holds true with opposite signs on $\{z<0\}\times \mathcal K^+$:
\begin{equation}\label{eq:cas - 2}
(\forall (v_i,v_j)\in \mathcal K^+\times \mathcal K^+) \quad (v_i < v_j) \Rightarrow \dfrac{I(z)}{T_-^+}  > f(z,v_i) > f(z,v_j) \, . 
\end{equation}
The next Lemma describes the shape of the velocity profiles on $\mathcal K^+$ for $z>0$. 

\begin{lemma}[Qualitative behaviour of the stationary density]\label{lem:mon}
Assume that $I$ is non increasing on $\{z>0\}$, and non decreasing on   $\{z<0\}$. Let $z>0$. For all $v_i>c$, we face the following alternative: either $T_+^+ f(z,v_i)\geq I(z)$, or for all $v_j>v_i$, $f(z,v_j) < f(z,v_i)$. Alternatively speaking, on the right hand side $(z>0)$, $f$ is decreasing with respect to velocity on the subset $\mathcal K^+\cap \{T_+^+ f< I\}$. 
\end{lemma}

\begin{proof}
We introduce the notation $\mathfrak{f} = T f$. 
From \eqref{eq:duhamel2}, we deduce
\begin{multline}
\mathfrak f(z+h,v_k) - I(z+h) = \left ( \mathfrak f(z,v_k) - I(z) \right ) \exp\left (  -\frac{ h T_+^+}{v_k-c} \right ) \\ + \left ( I(z) - I(z+h) \right ) \exp\left (  -\frac{ h T_+^+}{v_k-c} \right )
\\  +  \int_0^{\frac{h}{v_k-c}} \left (I(z + h -s(v_k-c)) - I(z+h)\right ) \exp\left ( -s T_+^+ \right )\, .\label{eq:decay space}
\end{multline}
As a consequence, if $\mathfrak f(z,v_k) > I(z)$, then  for all $h>0$, $\mathfrak f(z+h,v_k) > I(z+h)$, because the last two contributions in \eqref{eq:decay space} are non negative by assumption. On the contrary, if $\mathfrak f(z,v_k) < I(z)$, then  for all $y\in(0,z)$, $\mathfrak f(z-y,v_k) < I(z-y)$.

To characterize the monotonicity with respect to velocity, we get back to the kinetic transport formulation \eqref{eq:TW}. Let $v_i>c$ such that $\mathfrak f(z,v_i) < I(z)$. Let $v_j>v_i$. We have the following differential equality,
\begin{equation*}
(v_j-c)\partial_z\left( f(z,v_j) - f(z,v_i)\right ) + (v_j-v_i)\partial_z f(z,v_i) = - T_+^+\left ( f(z,v_j) -  f(z,v_i)\right )\, .
\end{equation*}
Denoting by $g(z) = \frac{f(z,v_j) - f(z,v_i)}{v_j-v_i}$, we have:
\begin{equation*}
(v_j-c) \frac{d g}{dz} (z) + T_+^+ g(z) = - \partial_z f(z,v_i) =   \dfrac{1}{v_i-c}\left ( \mathfrak f(z,v_i) -  I(z) \right )\,.
\end{equation*}
We deduce,
\begin{multline*}
g(z) = g(0) \exp\left (  -\frac{ z T_+^+}{v_j-c} \right ) \\
 + \dfrac{1}{v_i-c} \int_0^{\frac{z}{v_j-c}}  \left [ \mathfrak f(z-s(v_j-c),v_i) - I(z-s(v_j-c)) \right ] \exp\left ( -s T_+^+ \right )\, ds\, .
\end{multline*}
Both contributions in the right hand side are negative. Indeed, we have $g(0)<0$ by \eqref{eq:cas - 2}, and  $\mathfrak f(z-y,v_i) < I(z-y)$ for $y\in(0,z)$ since $\mathfrak f(z,v_i) < I(z)$. As a conclusion, we have
\begin{equation*}
(\mathfrak f(z,v_i) < I(z)) \quad \Rightarrow \quad (\forall v_j>v_i)\quad f(z,v_j) < f(z,v_i)\, .
\end{equation*}
\qed\end{proof}

The properties stated in Lemma \ref{lem:mon} enable to decipher the compensations in the velocity average $\rho(z) = \sum_{k\in \mathcal K} \omega_{k}   f(z,v_{k})$, yielding appropriate monotonicity for $\rho$.

\begin{lemma}[Enhancement of monotonicity]\label{lem:mon2}
Assume that $I$ is non increasing on $\{z>0\}$, and non decreasing on   $\{z<0\}$. Then both $\rho^+$ and $\rho^-$ are decreasing on  $\{z>0\}$, and both $\rho^+$ and $\rho^-$ are increasing on  $\{z<0\}$. 
\end{lemma} 

\begin{proof}
We compute first the derivative of $\rho^-$ on $\{z>0\}$:
\begin{equation*}
\dfrac{d\rho^-}{dz}(z) =  \sum_{k\in \mathcal K^-} \omega_k \partial_z f(z,v_k)  =  \sum_{k\in \mathcal K^-} \omega_k \dfrac{I(z) - T_+^- f(z,v_k)}{v_k-c}\, . 
\end{equation*}
We deduce from \eqref{eq:cas -} that each term in the sum is negative, because $v_k<c$ there. 
On the other hand, the derivative of $\rho^+$ reads as follows,
\begin{equation*}
\dfrac{d\rho^+}{dz}(z) =  \sum_{k\in \mathcal K^+} \omega_k \partial_z f(z,v_k)  =  \sum_{k\in \mathcal K^+} \omega_k \dfrac{I(z) - T_+^+ f(z,v_k)}{v_k-c}\, . 
\end{equation*}
The key observation is that we can omit the decreasing weights $(v_k-c)^{-1}$ in the last sum: 
\begin{equation*}
\dfrac{d\rho^+}{dz}(z) \leq  \sum_{k\in \mathcal K^+} \omega_k \left (I(z) - T_+^+ f(z,v_k)\right ) = - \sum_{k\in \mathcal K^-} \omega_k \left (I(z) - T_+^- f(z,v_k)\right ) < 0 \, . 
\end{equation*}
The last identity is a consequence of the very definition of $I$ \eqref{eq:def I 2}. In order to establish the inequality, we notice that the cumulative sum 
\begin{equation*}
H_k(z) = \sum_{j = J+1}^k \omega_j   \left (I(z) - T_+^+ f(z,v_j)\right ) \,,
\end{equation*}
satisfies the following properties:
\begin{equation*}
H_{J}(z) = 0\, , \quad \text{and}\quad (\forall k>J)\quad H_k(z) < 0\, .
\end{equation*}
Indeed, for fixed $z>0$, the sequence $\left (\omega_j   \left (I(z) - T_+^+ f(z,v_j)\right )\right )_{j\in \mathcal K^+}$ has the following pattern: the terms are first negative, then possibly positive. The reason is that the sequence is increasing as soon as it becomes positive (Lemma \ref{lem:mon}). Hence, the sequence $(H_k(z))_{k\in \mathcal K^+}$ has the following pattern: it is first decreasing, then possibly increasing. 
As the overall sum $H_K(z)$ is negative by the very definition of $I$, we deduce that each intermediate cumulative summation gives a negative value: for all $k\in \mathcal K^+$, $H_k(z)<0$. Consequently, we have after summation by parts,
\begin{align*}
\dfrac{d\rho^+}{dz}(z)  & =  \sum_{k = J+1}^K \dfrac{1}{v_k-c} \left ( H_{k} - H_{k-1}\right ) \\
& = \sum_{k = J+1}^{K-1} \left (\dfrac{1}{v_k-c} - \dfrac{1}{v_{k+1}-c}\right )   H_{k}(z)  + \dfrac{H_{K}(z)}{v_K-c} < 0\, . 
\end{align*}  
\qed\end{proof}

Lemma \ref{lem:mon2} is very useful to prove the monotonicity result stated in Theorem \ref{th:mono}. To this end, we make the weights $(\omega_k)$ vary continuously from any initial configuration $ \boldsymbol\omega^0 = \omega = (\omega_k)$ to the final state $\boldsymbol\omega^1 = \left (\frac12,0,\dots,0,\frac12\right )$. For $s\in[0,1]$, let define $f^s$ the solution of \eqref{eq:TW} associated with the weight $\boldsymbol\omega^s = (1-s) \boldsymbol\omega^0 + s \boldsymbol\omega^1$, having normalized mass \eqref{eq:unit}. 

Standard arguments enable to prove that the map $s\mapsto f^s$ is continuous for the topology of uniform convergence: Firstly,  the solution of \eqref{eq:TW} is unique for a given set of weights. Secondly, the function is uniformly bounded and Lipschitz continuous. Lastly, it is uniformly small outside a compact interval $[-L,L]$, see Section \ref{sec:confinement}. 

These statements require some justifications, based on the results established in Section \ref{sec:confinement}. As discussed previously, we are able to prove in a quantitative way the asymptotic behaviour \eqref{eq:asympt 1}, provided that we can bound $b_K$ and $\lambda_K(c)$ from below, as well as the spectral gap $\lambda_{K-1}(c)-\lambda_K(c)>0$. The bounds on  $b_K$ and $\lambda_{K-1}(c)-\lambda_K(c)$ rely on the non degeneracy of the extremal weight $\omega_K$. This is guaranteed uniformly along the sequence of weights $\omega^s$, as we have, by definition, 
\begin{equation*}
(\forall s\in [0,1])\quad \omega^s_K \geq \min\left (\omega_K, \frac12\right )\, .
\end{equation*}
The bound on $\lambda_K(c)$ requires some additional argument, related to the confinement property. Indeed, the confinement by biased velocity-jump processes is equivalent to the inequality $c<c^\star$. However, $c$ is fixed here, but $c^\star$ implicitly depends on the weights $\omega^s$, which are not constant. Thus, it is mandatory to guarantee that confinement holds true along the sequence of weights $(\omega^s)$. This is the purpose of the next lemma. Let us emphasize that we use here the condition $\chi_N\leq \chi_S$ \eqref{eq:cond chi}. We believe this restriction is not needed to ensure the final result. However, it simplifies the proof, as the sequence of weights must be redefined in a delicate way if $\chi_N>\chi_S$. In particular, the final state should not charge the extremal velocities $v_{\pm K}$, but rather some intermediate weights.

\begin{lemma}[Uniform confinement along the sequence]
\label{lem:confinement}
Under condition \eqref{eq:cond chi}, the eigenvalue $\lambda_K(c)$ is uniformly bounded from below, uniformly for $s$ in $[0,1]$. 
\end{lemma}

\begin{proof}
As $s$ belongs to a compact interval, it is sufficient to establish that $\lambda_K(c)$ remains positive for all $s$. 
The confinement on the right hand side is guaranteed for all $s$, provided that the mean algebraic run length is negative \eqref{eq:confinement +}:
\begin{equation*}
\sum_{k\in\mathcal K} \omega^s_k \dfrac{v_k-c}{T_+(v_k-c)}<0\,.
\end{equation*}
As this expression is linear with respect to $s$, and negative at $s=0$ by assumption, it is sufficient to check that the final value at $s=1$ is smaller than the initial value at $s=0$. Hence, we are reduced to establish the following inequality:
\begin{equation}\label{eq:comparison +}
\frac12 \left ( \dfrac{v_K-c}{T_+^+} + \dfrac{v_{-K}-c}{T_+^-} \right )
\leq  \sum_{k\in\mathcal K} \omega_k \dfrac{v_k-c}{T_+(v_k-c)}
\end{equation} 
We may recombine this inequality by factoring out $1/T_+^+$, on the one hand, and $1/T_+^-$, on the other hand:
\begin{equation*}
\dfrac{1}{T_+^+}\left ( \frac12 (v_K-c) - \sum_{k>J} \omega_k (v_k-c)\right ) \leq  \dfrac{1}{T_+^-} \left (  - \frac12 (v_{-K}-c) + \sum_{k\leq J} \omega_k (v_k-c)\right )\, .
\end{equation*}
We claim that both factors are equal and non negative:
\begin{equation*}
\frac12 (v_K-c) - \sum_{k>J} \omega_k (v_k-c) = - \frac12 (v_{-K}-c) + \sum_{k\leq J} \omega_k (v_k-c) \geq  0\, .
\end{equation*} 
This is an immediate consequence of $\sum \omega_k  = 1$, and  $\sum \omega_k v_k = 0$. The positive sign can be viewed easily on the left hand side: $ \sum_{k>J} \omega_k (v_k-c) \leq  (v_K-c) \sum_{k>J} \omega_k \leq (v_K-c)/2$. We conclude that \eqref{eq:comparison +} holds true, since $T_+^- < T_+^+$ by \eqref{eq:cond chi} and \eqref{eq:sign T}. 

The same arguments lead to the opposite inequality on the left hand side, in order to ensure uniform confinement there as well:
\begin{equation*}
(\forall s)\quad \sum_{k\in\mathcal K} \omega^s_k \dfrac{v_k-c}{T_-(v_k-c)}>0\,.
\end{equation*}
\qed\end{proof}

We are now in position to conclude the proof of Theorem \ref{th:mono}. Lemma \ref{lem:asympt mon} is useful to compactify the space interval, since the appropriate monotonicity is guaranteed for $|z|>L$, uniformly with respect to $s\in [0,1]$. Let $\mathcal S$ be the set of values $s$ for which $I^s$ has the appropriate monotonicity:
\begin{equation}\label{eq:appropriate mono}
\mathcal S = \left\{ s\in [0,1] \;|\; (\forall z\neq 0)\;  (\mathrm{sign}\; z)\dfrac{dI^s}{dz}(z) \leq 0  \right\} \, .
\end{equation}
It is immediate to see that $\mathcal S$ contains the extremal value $s=1$, simply because the eigenmode decompositions \eqref{eq:decomposition} are reduced to a single element on each side. Alternatively speaking, the solution is an exponential function on each side, so it is monotonic. Confinement (Lemma \ref{lem:confinement}) guarantees that the exponential functions decay on both sides, so it has the appropriate monotonicity \eqref{eq:appropriate mono}. 

Theorem \ref{th:mono} is equivalent to say that the value $s=0$ belongs to $\mathcal S$, also. We argue by connectedness. 
\begin{itemize}
\item Firstly, $\mathcal S$ is open. Let $s_0\in \mathcal S$. Lemma \ref{lem:mon2} says that monotonicity is strict on both sides $\{z<0\}$, and $\{z>0\}$. We claim that there exist $\varepsilon_0>0$, and  a neighbourhood $\mathcal V_0$ of $s_0$ in $\mathcal S$, such that 
\begin{equation*}
(\forall s\in \mathcal V_0)\; (\forall z\in[-L,L]\setminus\{0\})\quad  (\mathrm{sign}\; z)\dfrac{dI^s}{dz}(z) < -\varepsilon_0\, .
\end{equation*}
We argue by contradiction: otherwise, there would exist a sequence $s_n\to s_0$, and a sequence $z_n$ with constant sign, converging towards some $z_0\in [-L,L]$, including value 0, such that   
\begin{equation*}
\lim_{n\to +\infty}(\mathrm{sign}\; z_n)\frac{dI^{s_n}}{dz}(z_n)= 0\, .
\end{equation*}
We can pass to the limit in the sequence of derivative functions $\frac{dI^{s_n}}{dz}$, uniformly over $[-L,0]$ or $[0,L]$\footnote{Notice that the value $z = 0$ is not an issue here. In fact, Lemma \ref{lem:mon2} includes the values $z = 0^+$ and $z = 0^-$ which must be distinguished from each other because $I$ is not continuous at $z = 0$.}. This yields a contradiction with the strict monotonicity at $s=s_0$. 
\item
Secondly, $\mathcal S$ is closed. This is an easy consequence of the continuity of the sequence $s\mapsto f^s$, for the topology of uniform convergence, which is   compatible with monotonicity properties.  \qed
\end{itemize}
\end{proof}

\section{Coupling with the reaction-diffusion equations: three examples}
\label{sec:EX}

As soon as the existence of a function $f$ is established for a given $c$, as in Theorem \ref{th:exist}, with appropriate monotonicity properties as in Theorem \ref{th:mono}, \textbf{the coupling with the reaction-diffusion equations through $\rho$ in the travelling wave problem \eqref{eq:TW} is essentially not sensitive to the topology of the velocity space} (discrete as in here, or continuous as in \cite{calvez_chemotactic_2016}). Therefore, we do not repeat the analysis performed in \cite{calvez_chemotactic_2016}, but we recall two useful propositions.

\begin{proposition}\label{prop:monotonicity S}
Assume that the function $\rho\in L^1$ is locally Lipschitz continuous on $\mathbb{R}^*$, and that it is increasing for $z<0$, and decreasing for $z>0$. Let $S$ be the unique solution of the following elliptic problem,
\[ (\forall z\in \mathbb R) \quad -c\partial_z S(z) - D_S \partial^2_z S(z) + \alpha S(z) = \beta \rho(z)\, . \]
Then $S$ is unimodal, meaning that $\partial_z S$ vanishes once, and only once.
\end{proposition}

\begin{proposition}\label{prop:N}
Assume that $c> 0$, and that $\rho$ is exponentially bounded on both sides, namely
\begin{equation*}
\begin{cases}
(\forall z<0) \quad & \rho(z) \leq C \exp\left (\lambda_- z\right )\medskip\\
(\forall z>0) \quad & \rho(z) \leq C \exp\left ( -\lambda_+ z\right )
\end{cases}
\end{equation*}
Then, there exist two constants $N_-,N_+$, ordered as $0<N_-<N_+$, and a solution $N$ of the following elliptic problem 
\begin{equation*}
- c\partial_z N(z)  - D_N \partial_{z}^2 N(z) = - \gamma \rho(z) N(z)\,,
\end{equation*}
such that 
\[
\begin{cases}
\displaystyle\lim_{z\to -\infty} N(z) = N_-\medskip\\
\displaystyle\lim_{z\to +\infty} N(z) = N_+
\end{cases}
\]
Moreover, for all $z$ in $\mathbb{R}$, $\partial_z N(z)>0$.  
\end{proposition} 

We deduce from these two propositions that the framework emphasized in the Introduction is consistent, except that in point (i) it is required that the maximum of $S$ is located at $z=0$. The latter requirement is not true in general. In fact, it is intuitively clear that the location of the maximum depends on the value of $c$, which is the main unknown of the problem \eqref{eq:TW}. 

The last step of the construction of travelling waves consists in varying $c$ to meet this last requirement about the location of the maximum at the origin. This is equivalent to say that the function $\Upsilon $ \eqref{eq:upsilon} has a root $\mathbf{c}$ in the interval of admissible velocities $(c_\star,c^\star)\cap(0,c^\star)$\footnote{The restriction $c>c_\star$ is due to confinement on the right hand side, whereas the condition $c> 0$ is due to the arbitrary choice of the direction of propagation of the wave (here, left to right), which influences itself the monotonicity of $N$ (here, increasing). It is  an arbitrary choice, of course, since the problem is symmetric.}. In the present work, we have  $(c_\star,c^\star)\cap(0,c^\star) = (0,c^\star)$, because $c_\star\leq 0$ under condition \eqref{eq:cond chi}. 
 
\begin{svgraybox}
Surprisingly enough, neither the existence nor the uniqueness of a root $\mathbf{c}$ such that $\Upsilon(\mathbf{c}) = 0$ do hold for any set of parameters (this is the main difference between the kinetic problem, and its macroscopic diffusive limit, as in \cite{saragosti_mathematical_2010,calvez_chemotactic_2016}). 
\end{svgraybox}

Here, we present three examples of possible shapes for $\Upsilon$, based on accurate numerical simulations. It is important to notice that the function $\Upsilon$ is not smooth in the case of discrete velocities (as opposed to the continuous velocity case, see \cite{calvez_chemotactic_2016}). In fact, it has  jump discontinuities  located on the set of velocities $ (v_k)_{k\in \mathcal K} $.  A thorough analysis of the sign and the size of these jump discontinuities was performed in \cite{calvez_chemotactic_2016}, questing for counter-examples. Here, we bypass this analysis, and we present directly the counter-examples.

\begin{figure}[t]
\begin{center}
\includegraphics[width = .8\linewidth]{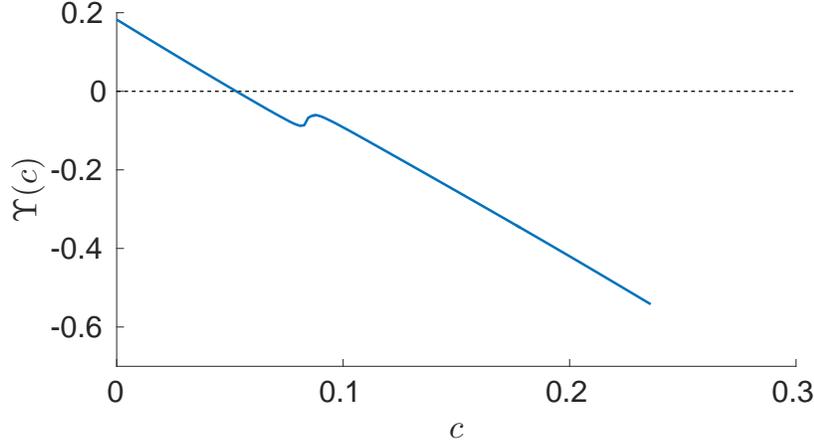}
\caption{Plot of the function $\Upsilon$ over $(0,c^\star)$ for the set of parameters described in Section \ref{sec:EX1}. We observe that the function is not globally monotonic. It is seemingly decreasing on each side of the positive jump discontinuity. However, this jump discontinuity is not large enough to imply the co-existence of two wave speeds. In this case, we observe the existence of a unique wave speed.}
\label{fig:upsilon-onewave}
\end{center}
\end{figure}

\subsection{Existence and uniqueness of the wave speed}  
\label{sec:EX1}

Here, we consider a first set of velocities and weights for which there is numerical evidence that $\Upsilon$ has a unique admissible root $\mathbf{c}$, as in Figure \ref{fig:upsilon-onewave}. 
The velocities are chosen as follows:
\begin{multline*}
v_1 =  0.0848 \,, v_2 =   0.2519  \,, v_3 =  0.4118 \,, v_4 =   0.5598 \,, v_5 =   0.6917 \,,\\ v_6 =   0.8037 \,, v_7 =   0.8926 \,, v_8 =   0.9558  \,, v_9 =  0.9916\,.
\end{multline*}
The weights are chosen as follows:
\begin{multline*}
w_1 =   0.0846  \,, w_2 =   0.0822 \,, w_3 =    0.0774 \,, w_4 =   0.0703   \,, w_5 =   0.0613\,, \\w_6 =    0.0505 \,, w_7 =   0.0382   \,, w_8 =   0.0249 \,, w_9 =     0.0108\,.
\end{multline*}
We set $\omega_0 = 0$ for numerical purposes. However, this does not affect the results of Section \ref{sec:confinement} and Section \ref{sec:mono}. 
The other parameters are: 
\begin{equation*}
\chi_S = 0.3\,, \chi_N = 0.15\,, D_S = 0.5\,, D_N = 1\,, \alpha = 0.5\,, \beta=1\,, \gamma=1\,.
\end{equation*}

\begin{figure}[t]
\begin{center}
\includegraphics[width = .8\linewidth]{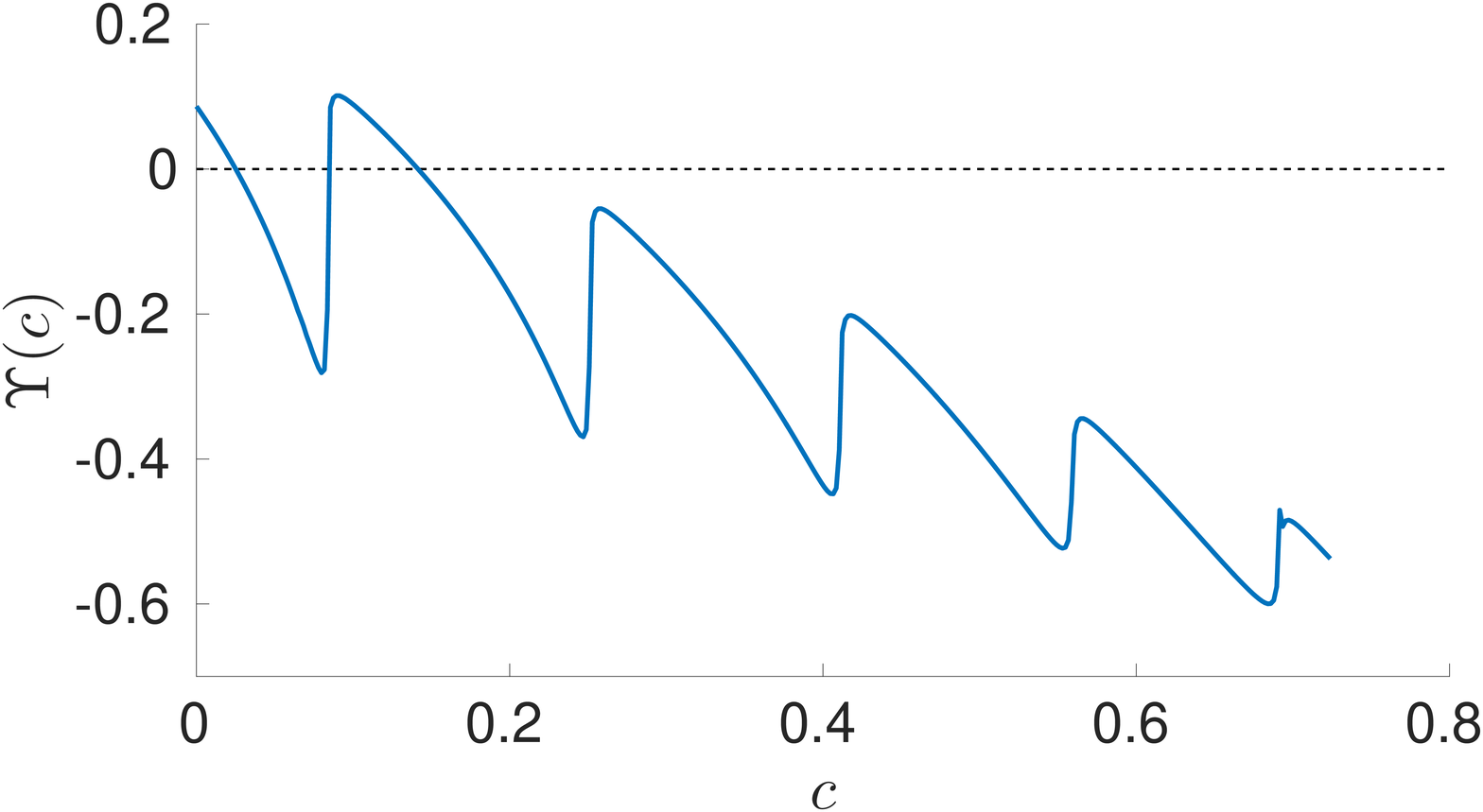}
\caption{Plot of the function $\Upsilon$ over $(0,c^\star)$ for the set of parameters described in Section \ref{sec:EX2}. We observe that the function $\Upsilon$ possesses two admissible roots, as it crosses the horizontal axis downwards. Note that the crossing upwards is not admissible as it corresponds to a jump discontinuity of $\Upsilon$.}
\label{fig:upsilon-multiplewave}
\end{center}
\end{figure}

\begin{figure}[t]
\begin{center}
\includegraphics[width = .45\linewidth]{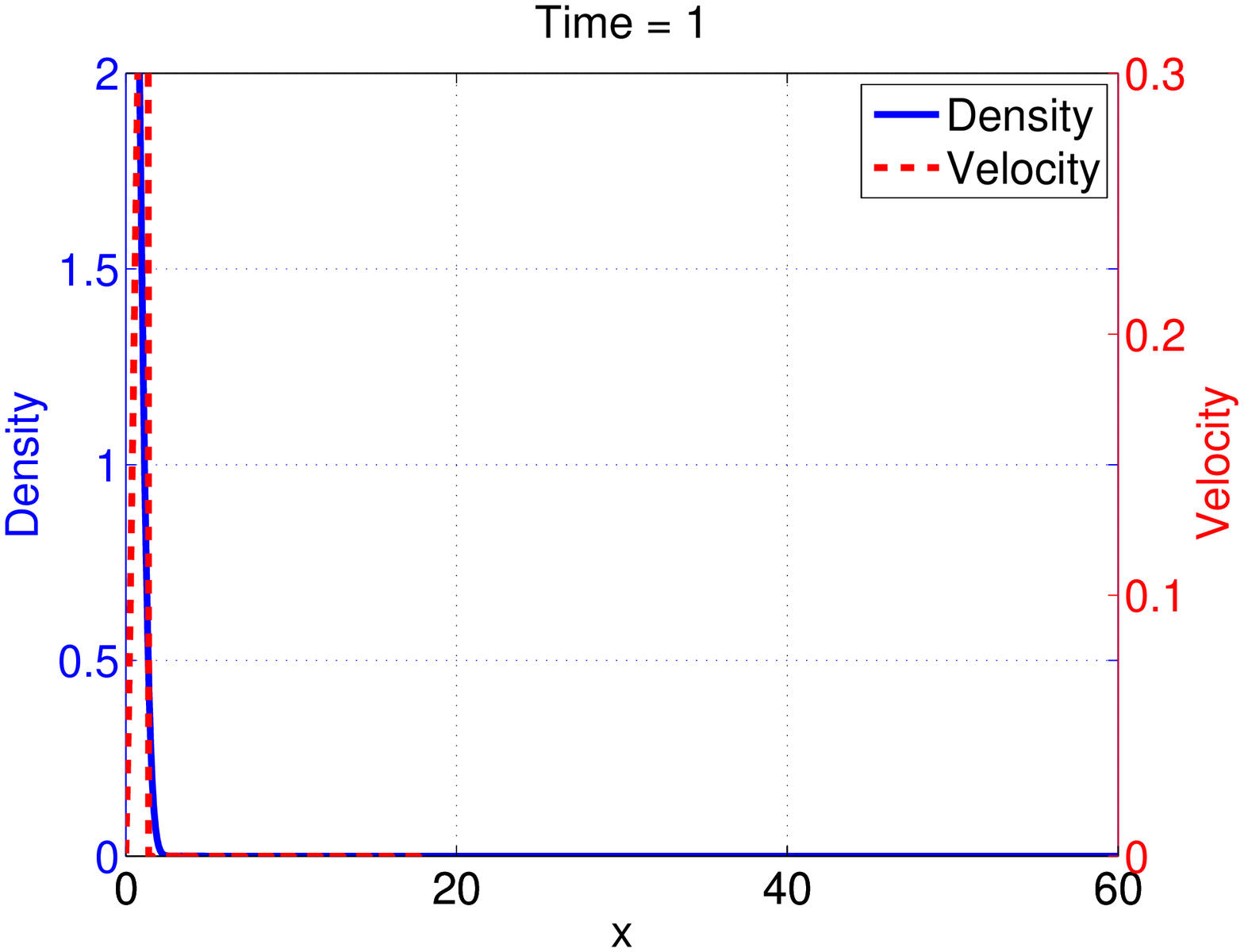}\;
\includegraphics[width = .45\linewidth]{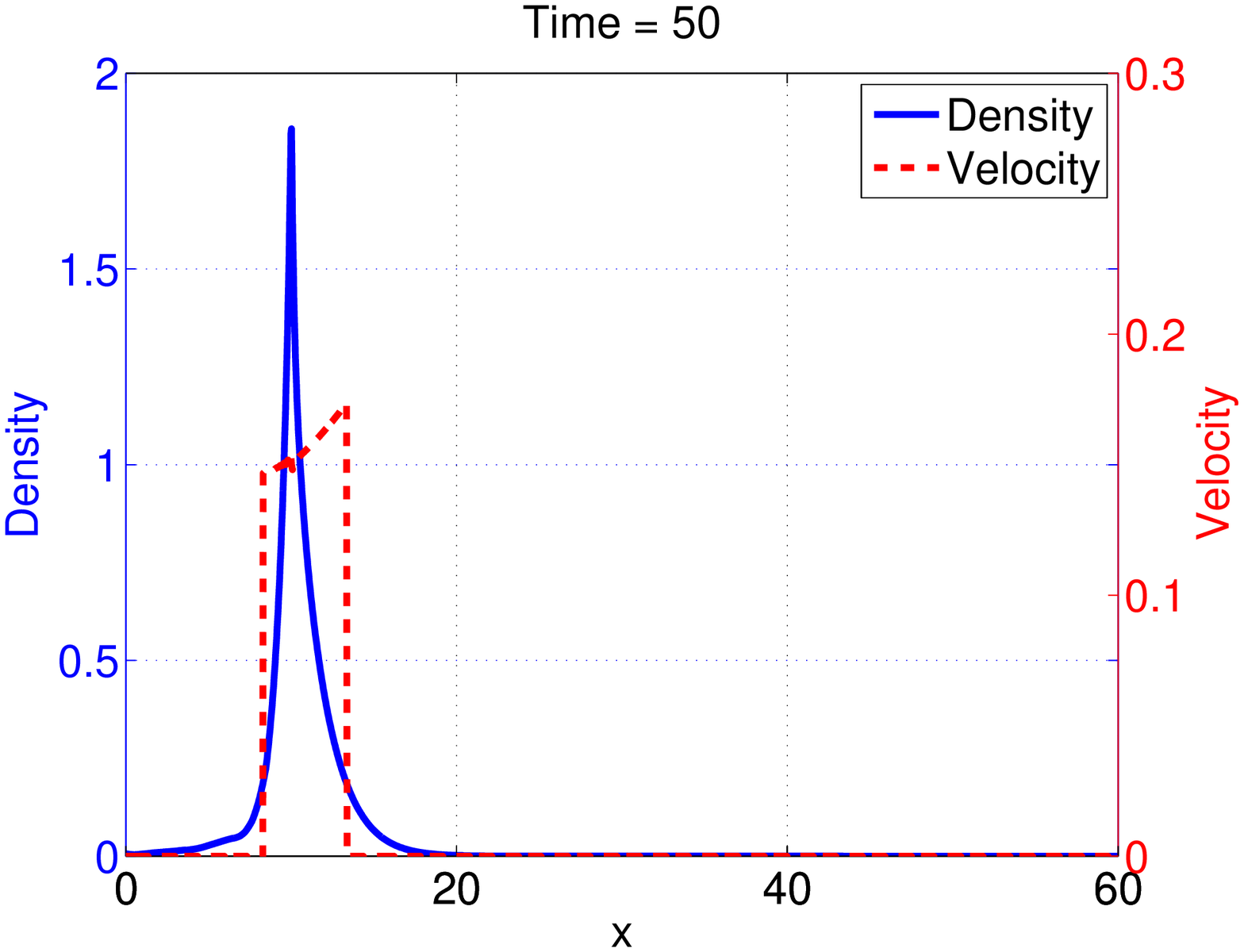}\\
\includegraphics[width = .45\linewidth]{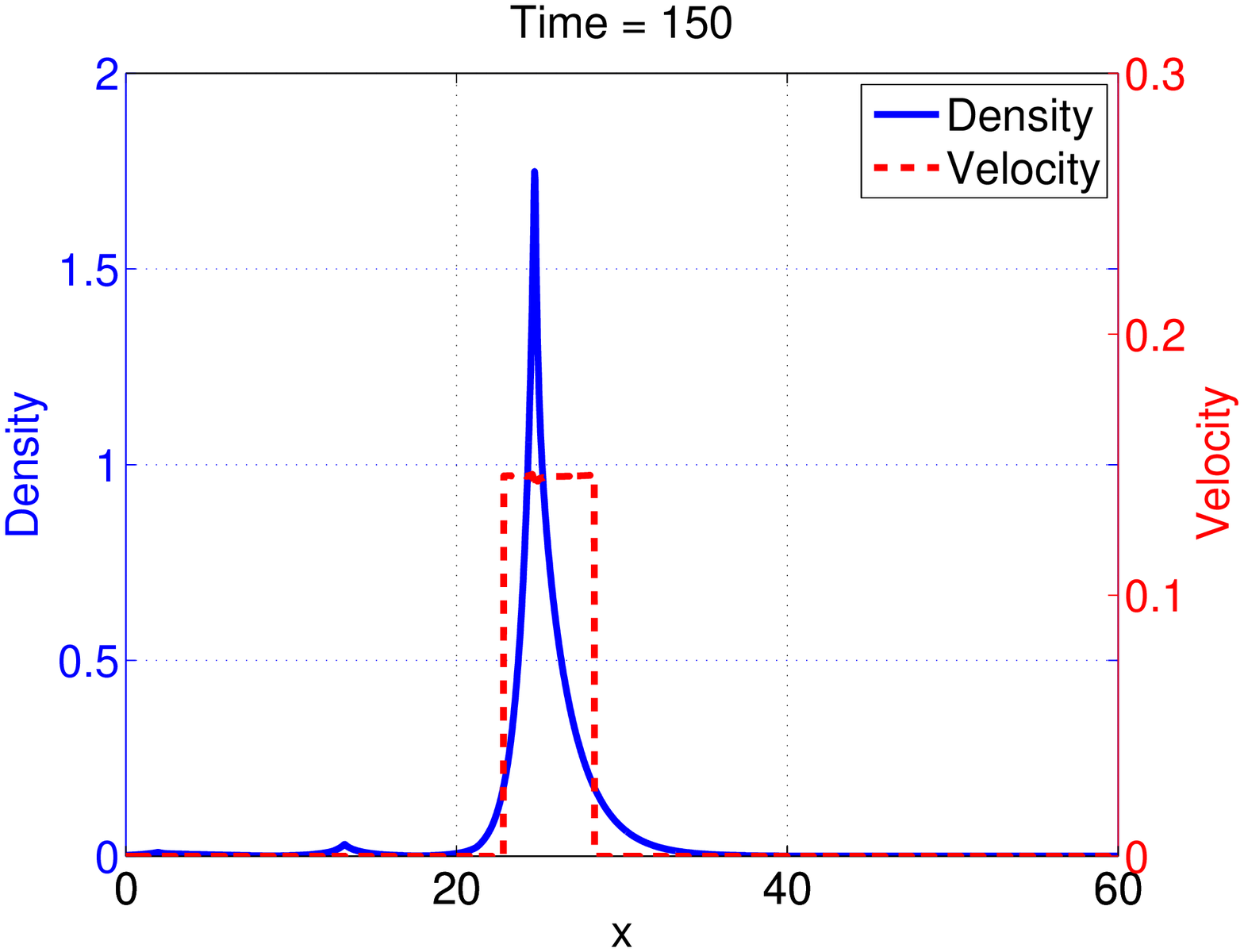}\;
\includegraphics[width = .45\linewidth]{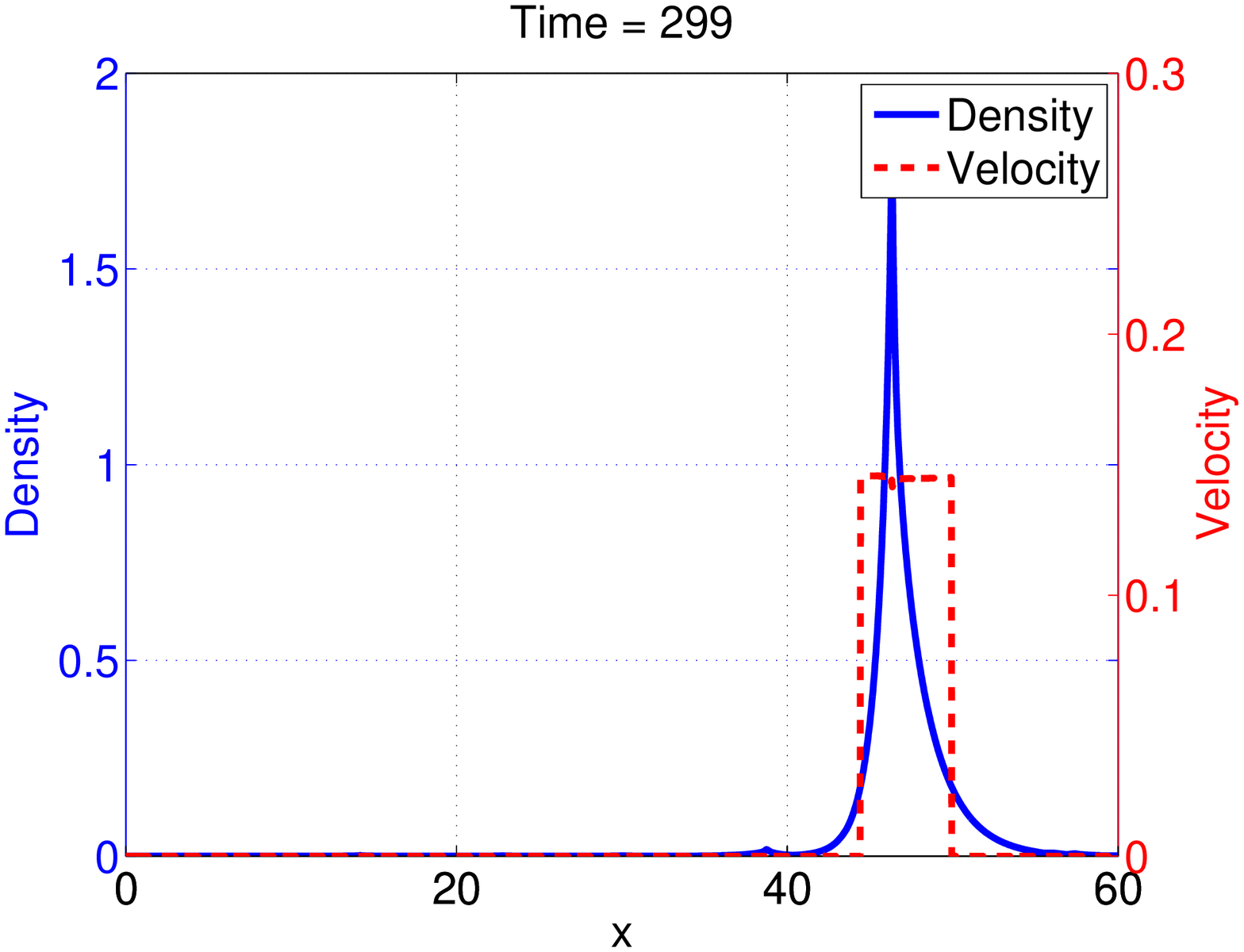}
\caption{Numerical simulations of the Cauchy problem in the case of multiple existence of travelling waves (second example in Section \ref{sec:EX2}). The formation of a stable   wave travelling at constant speed is clearly observed. It corresponds to the fastest wave: $\mathbf{c}\approx 0.15$ (compare with Figure \ref{fig:upsilon-multiplewave}).}
\label{fig:wave}
\end{center}
\end{figure}

\subsection{Non uniqueness of the wave speed}
\label{sec:EX2}

The second example consists in a set of velocities and weights, for which there is seemingly two possible wave speeds, as the function $\Upsilon$ crosses at least twice the zero axis on the set of admissible velocities, see Figure \ref{fig:upsilon-multiplewave}. There, we observe two crossing downwards, corresponding to two admissible wave speed. The vertical crossing upwards is not admissible as it corresponds to a jump discontinuity of $\Upsilon$ located at one of the discrete velocities $ (v_k)_{k\in \mathcal K} $. The parameters are the same as in Section \ref{sec:EX1}, except that the values of $\chi_S, \chi_N$ and $\alpha$ are replaced with: 
\[\chi_S = 0.5\,, \chi_N = 0.45\,, \alpha = 10\,.\]

Beyond the static analysis dealing exclusively with the construction of travelling wave solutions, we ran numerical simulations of the Cauchy problem \eqref{eq:meso model}. 
The well-balanced upwind numerical scheme is designed as follows: 
\begin{itemize}
\item The kinetic transport equation involves the setup of a formulation involving {\bf a scattering $S$-matrix at each cell's interface}, as presented in \cite{gosse_MMS} for slightly different kinetic models. Such a $S$-matrix is retrieved thanks to a Case's mode decomposition \eqref{eq:decomposition} associated with $c=0$, see \cite{companion} for details, \cite[Chapter 10]{gosse_computing_2013} and \cite{gosse_MBS,emako_well-balanced_2016} as well. 
\item Reaction-diffusion equations are treated by means of a ${\mathcal L}$-spline interpolation leading to accurate numerical fluxes where all the terms (diffusive, reactive, drift) can be treated as a whole, allowing for the preservation of a delicate balance between each other: see especially \cite{gosse_L}.
\end{itemize}

Results are shown in Figure \ref{fig:wave}.
Beginning with an initial data which is concentrated on the left-hand-side, with arbitrary shape, we observe the formation of a wave moving at constant speed over a reasonably long time span (Figure \ref{fig:wave}). The speed of propagation corresponds to the fastest of the two roots of $\Upsilon$, as seen on Figure \ref{fig:upsilon-multiplewave}.

A similar (but simpler) example of non uniqueness with only four velocities is extensively analysed in \cite{companion}. We conjecture that the two co-existing travelling waves are stable. Bistability is clearly shown using numerical simulations of the Cauchy problem with various initial data. 

\subsection{Non existence of the wave speed}
\label{sec:EX3}

\begin{figure}[t]
\begin{center}
\includegraphics[width = .8\linewidth]{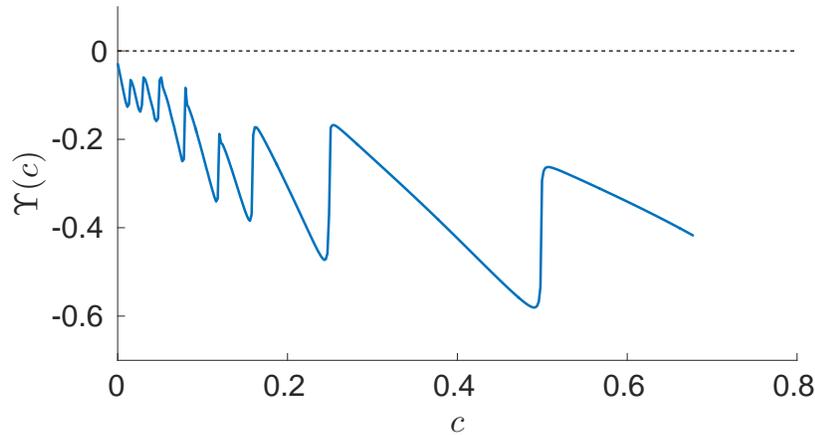}
\caption{Plot of the function $\Upsilon$ over $(0,c^\star)$ for the set of parameters described in Section \ref{sec:EX3}. We observe that the function is negative for all admissible $c$. This means that the maximum of the concentration $S$ is always located on the left side of the origin. Consequently, the construction of travelling wave cannot be achieved in this case.}
\label{fig:upsilon-nowave}
\end{center}
\end{figure}

\begin{figure}[t]
\begin{center}
\includegraphics[width = .45\linewidth]{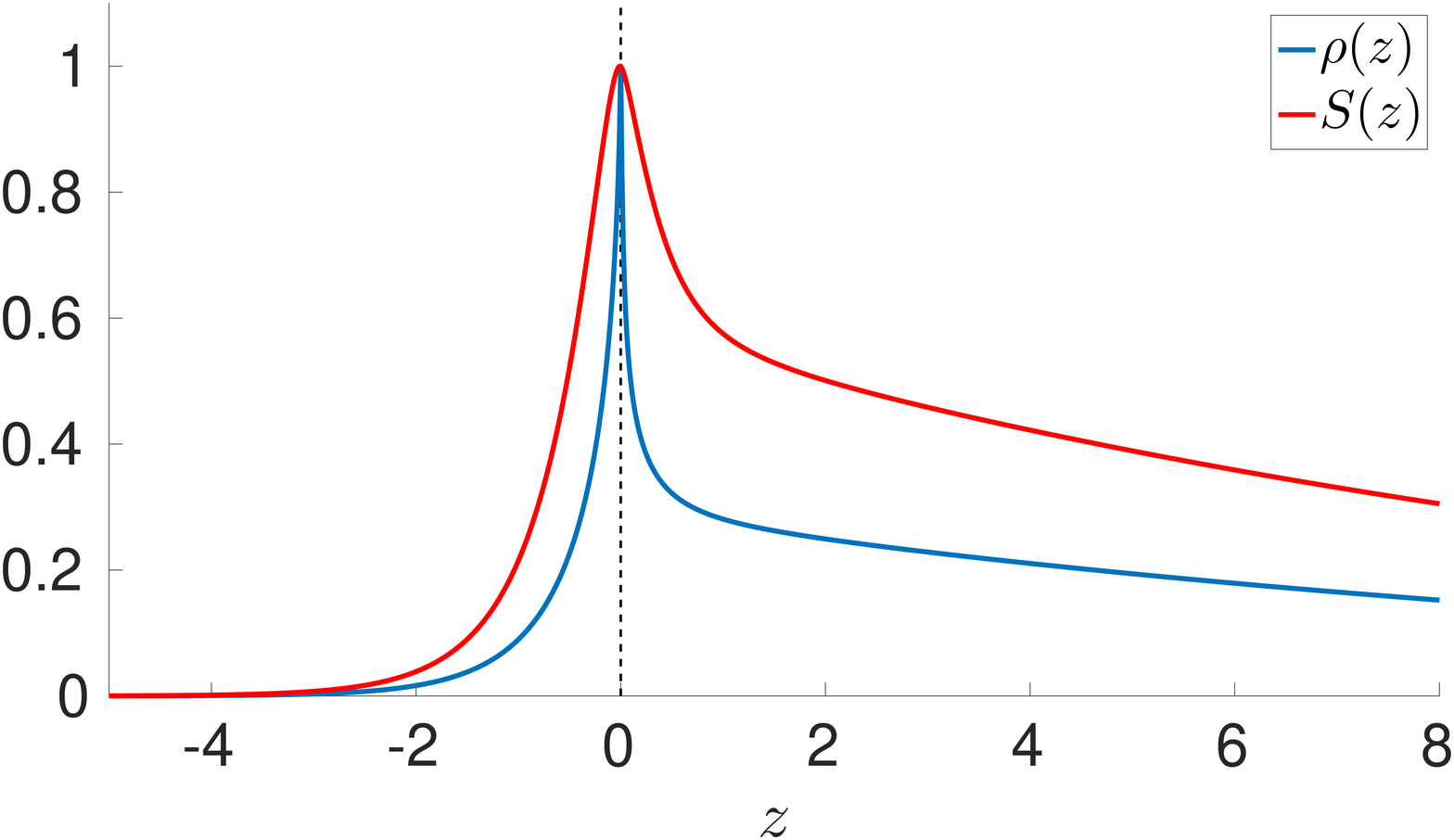}\;
\includegraphics[width = .45\linewidth]{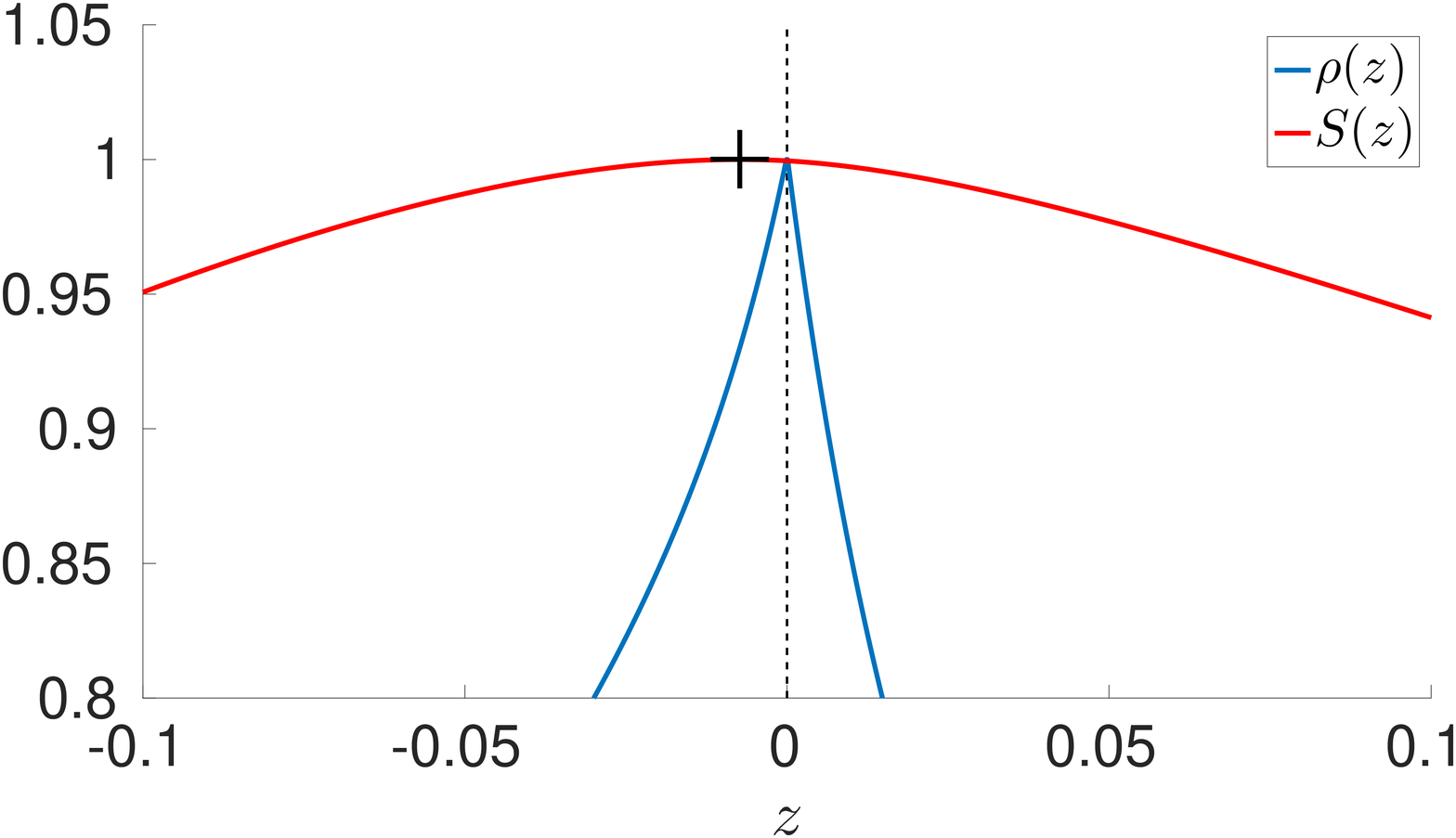}
\caption{Illustration of the counter-inutuitive phenomenon occurring in the third example (Section \ref{sec:EX3}). For $c = 0$, the cell density $\rho$ is tilted to the right side for large $|z|$, because the positive gradient of nutrient yields net biased motion in the right direction (compare $\rho(z)$ and $\rho(-z)$ for large $z\gg 1$). However, it is tilted to the left side for small $|z|\ll 1$. As a consequence, the chemical concentration $S$ reaches its maximum at a negative value, for a suitable choice of the reaction-diffusion parameters (see the zoom on the right). This prevents the existence of a travelling wave.}
\label{fig:counter-intuitive}
\end{center}
\end{figure}

\begin{figure}[t]
\begin{center}
\includegraphics[width = .45\linewidth]{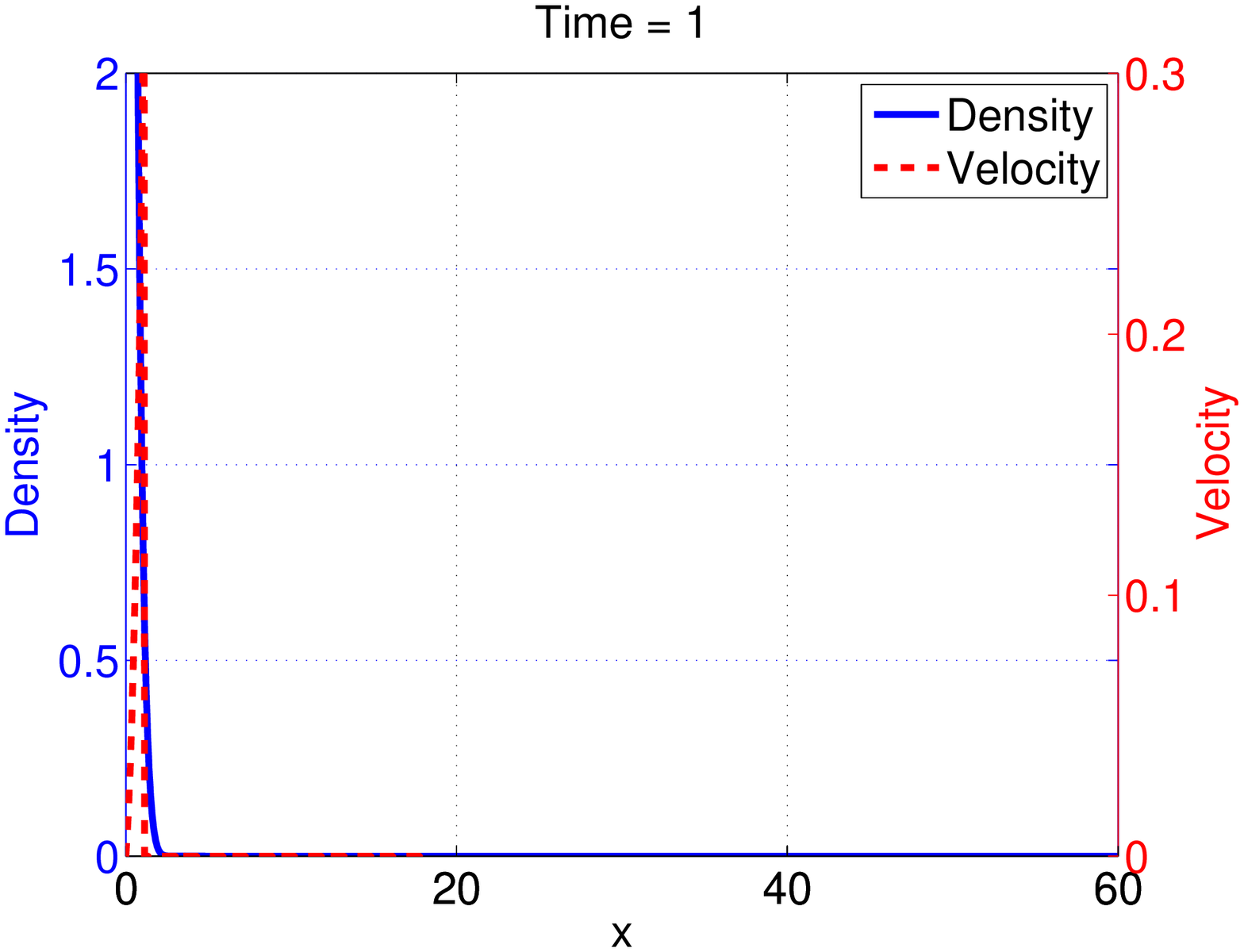}\;
\includegraphics[width = .45\linewidth]{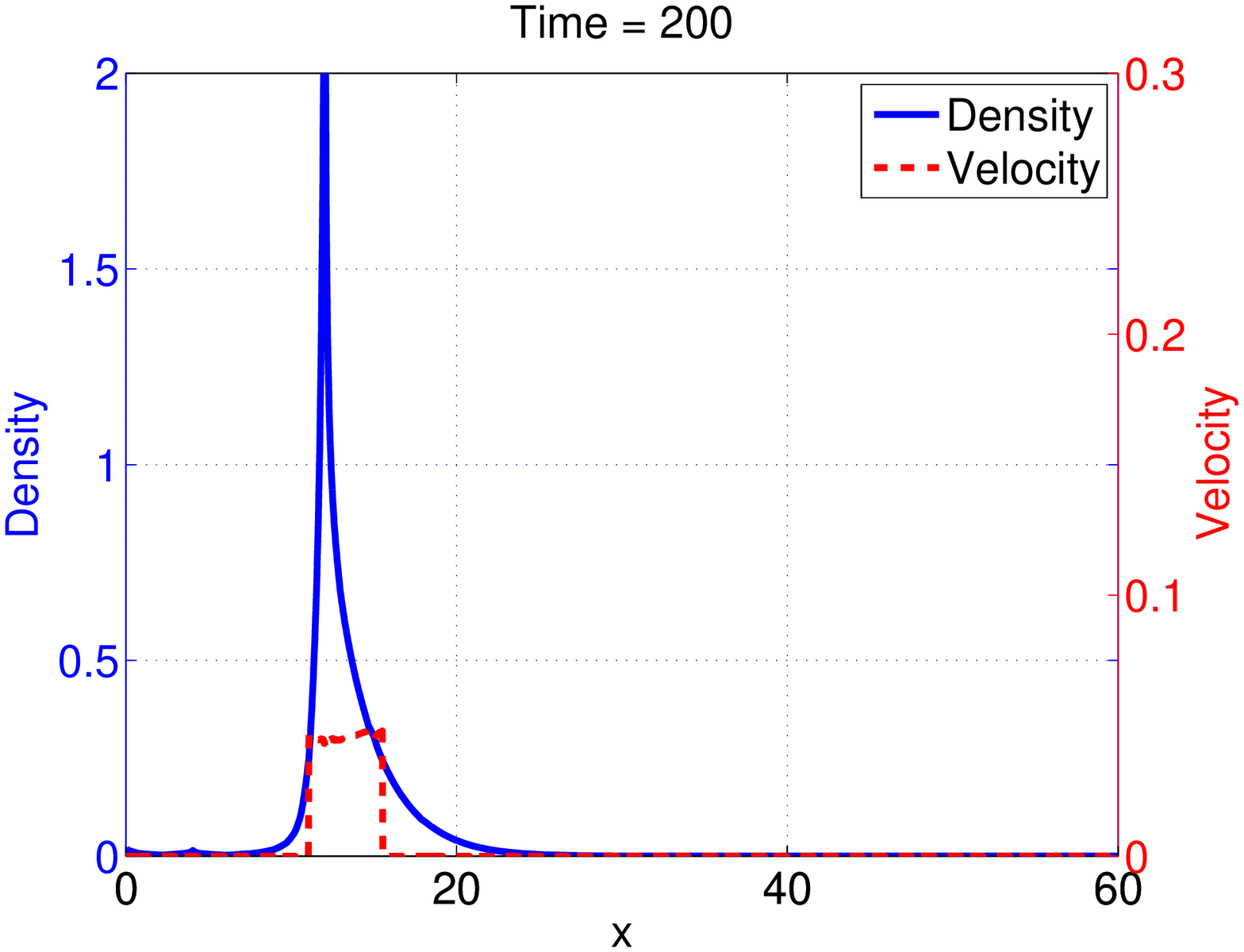}\\
\includegraphics[width = .45\linewidth]{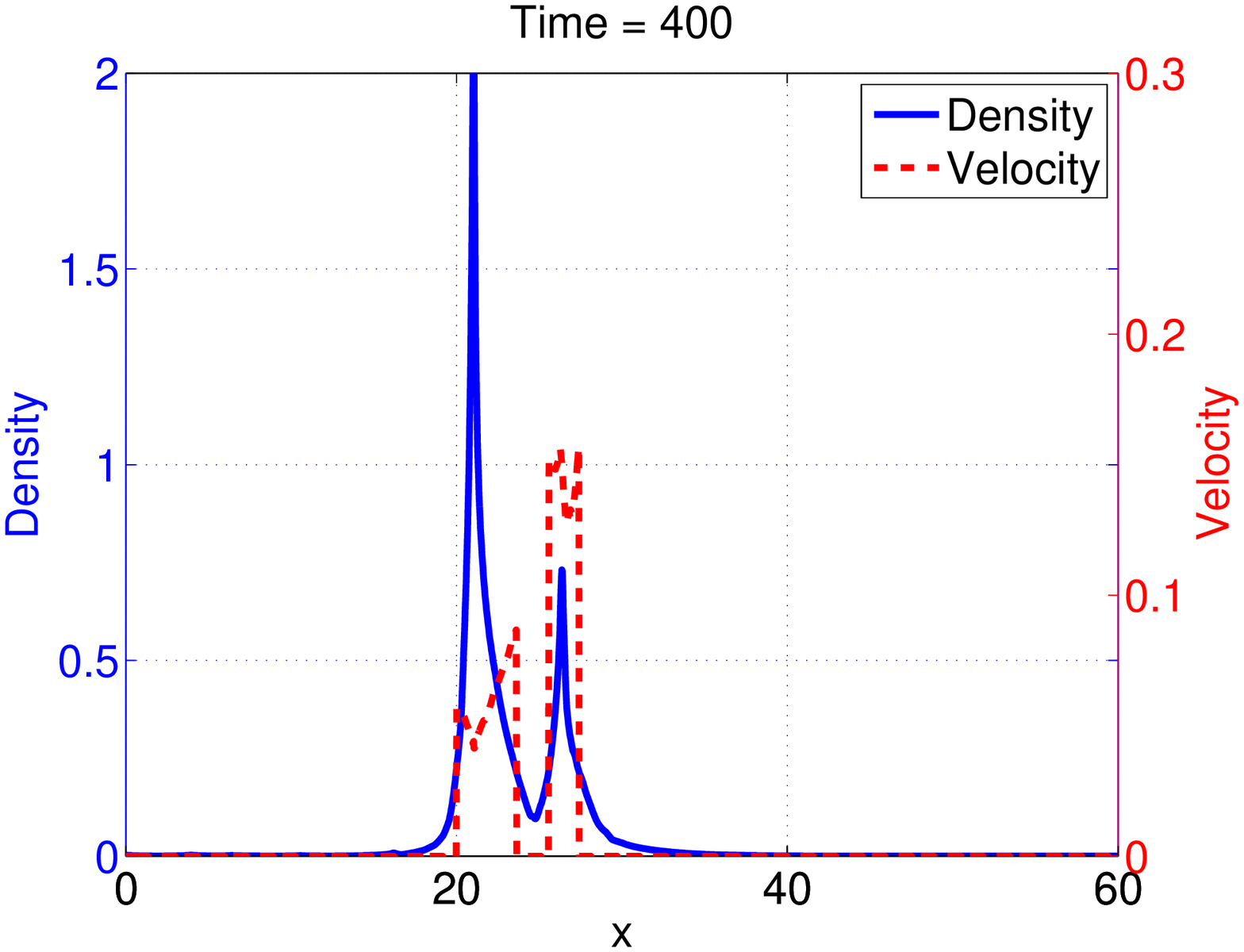}\;
\includegraphics[width = .45\linewidth]{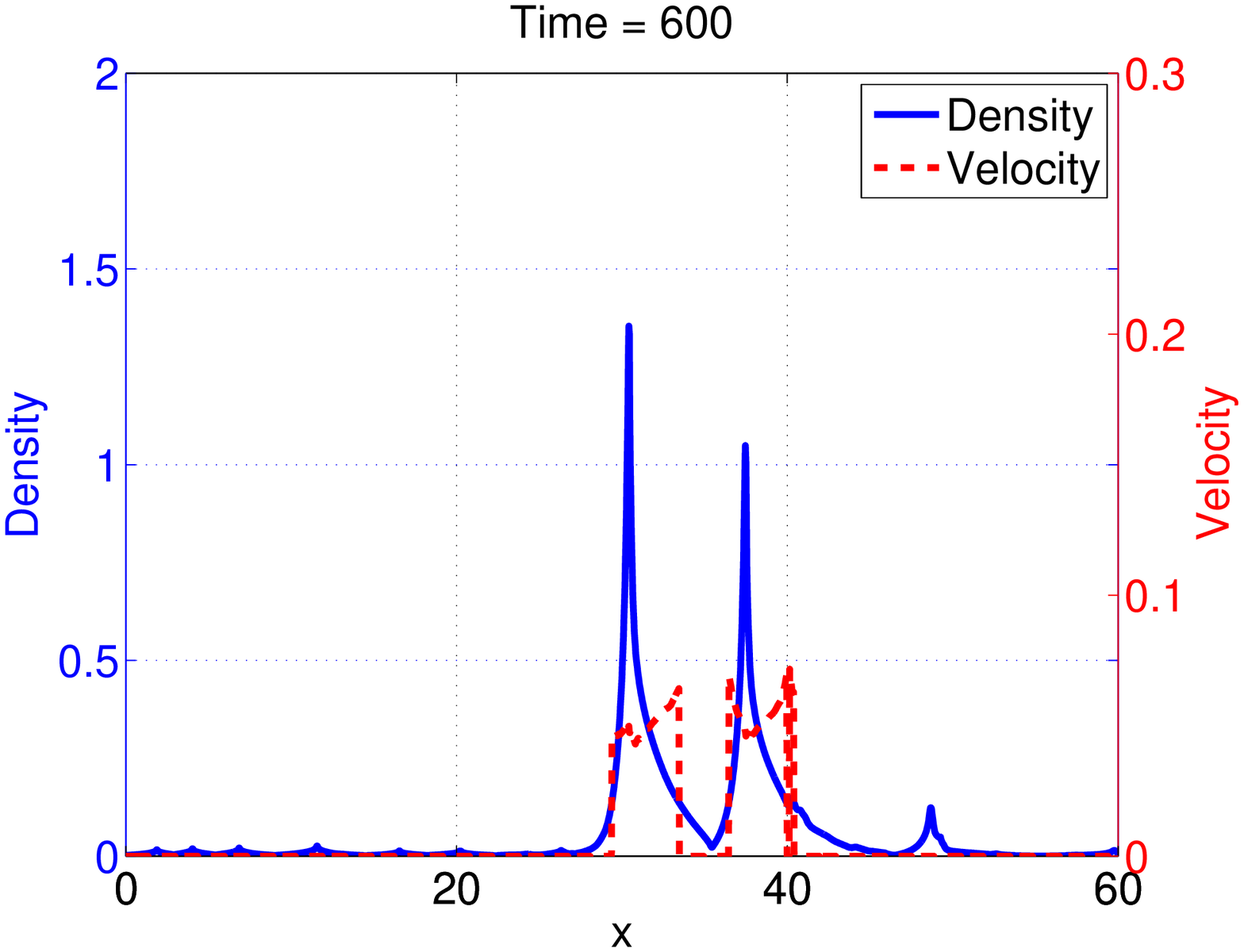}
\caption{Numerical simulations of the Cauchy problem in the case of non existence of a  travelling wave (third example in Section \ref{sec:EX3}). As opposed to Figure \ref{fig:wave}, the cell density quickly splits into smaller units. However, we still observe global propagation of the cell population to the right side.}
\label{fig:nowave}
\end{center}
\end{figure}

The third example consists in a set of velocities and weights, for which there is seemingly no admissible wave speed, as in Figure \ref{fig:counter-intuitive}. The velocities are chosen as follows:
\begin{multline*}
v_1 = 0.015,\, v_2 =  0.03 \,,  v_3 =  0.05 \,,  v_4 =  0.08 \,,  v_5 =   0.12  \,,\\  v_6 =  0.16 \,,   v_7 =  0.25 \,,  v_8 =   0.5 \,,   v_9 =  1\,.
\end{multline*}
The weights are uniform: 
\begin{equation*}
(\forall k\in[-K,K]\setminus\{0\})\quad \omega_k = 1/18\, , \quad \omega_0 = 0\, .
\end{equation*}
The other parameters are as in Section \ref{sec:EX2}.

Let us emphasize that this is a counter-intuitive result. Indeed, it is highly related to the fact that $\Upsilon(0)<0$. So, let us focus on the case $c= 0$. We clearly have $\lambda_+(0) < \lambda_-(0)$. This is a way to express the net biased motion of cells to the right side, due to the contribution of the positive gradient of nutrient $\partial_z N>0$ in the tumbling rate \eqref{eq:T}, together with $\chi_N>0$. We could intuitively deduce that the spatial density $\rho$ is globally tilted to the right side (see Figure \ref{fig:counter-intuitive}). But, this is not true. It may happen that, for small $0<z\ll 1$, $\rho(z)< \rho(-z)$, meaning that the spatial density is locally tilted to the left side (see the zoom in Figure \ref{fig:counter-intuitive}). Then, by choosing appropriately the reaction-diffusion parameters $\alpha, D_S$, it is possible to transfer this local asymmetry to ensure $\Upsilon(0)<0$, meaning that the maximum of $S$ is located on the left side.

Beyond this negative result, it is interesting to run numerical simulations of the Cauchy problem \eqref{eq:meso model}. Indeed, we guess that the net biased motion to the right side makes the wave propagating in a way that is not compatible with the formation of a stable travelling wave. 
Results are shown in Figure \ref{fig:nowave}.
We observe the inclination to form a  wave moving to the right after short time, as expected intuitively. However, the cell density splits quickly into smaller components, as opposed to Figure \ref{fig:wave}, where it is maintained over the duration of the numerical test. This peculiar behaviour agrees  with the non existence of a travelling wave.


\paragraph{Acknowledgement}
This project has received funding from the European Research Council (ERC) under the European Union's Horizon 2020 research and innovation programme (grant agreement No 639638). M.T. has benefited from the PICS Project CNR-CNRS 2015-2017 "Mod\`eles math\'ematiques et simulations num\'eriques pour le mouvement de cellules".
\end{document}